\documentclass[reqno]{amsart}
\usepackage{amssymb,amsmath,amsthm,bm,mathrsfs,paralist}
\usepackage{url,xcolor,units,graphicx,mathtools,bbold,url,amsrefs}
\usepackage[T1]{fontenc}
\RequirePackage[colorlinks,citecolor=blue,urlcolor=blue]{hyperref}
\newcommand{\sC}{\mathscr{C}}
\newcommand{\sD}{\mathscr{D}}

\newcommand{\R}{\mathbb{R}}
\newcommand{\N}{\mathbb{N}}
\newcommand{\Z}{\mathbb{Z}}

\newcommand{\cG}{\mathscr{G}}
\newcommand{\cN}{\mathcal{N}}
\newcommand{\cM}{\mathcal{M}}
\newcommand{\cP}{\mathscr{P}}
\newcommand{\lip}{\text{\rm Lip}}
\DeclareMathOperator{\E}{\mathrm{E}}
\renewcommand{\d}{\mathrm{d}}
\renewcommand{\P}{\mathrm{P}}
\newcommand{\e}{\mathrm{e}}

\newcommand{\dW}{\dot{W}}
\DeclareMathOperator{\Cov}{\mathrm{Cov}}

\newcommand{\W}{\mathcal{W}}
\renewcommand{\L}{\mathcal{L}}
\usepackage[cal=dutchcal,calscaled=1.05,
	scr=boondoxo,scrscaled=1.05,
	]{%
mathalfa}
\newtheorem{proposition}{Proposition}
\newtheorem{theorem}[proposition]{Theorem}
\newtheorem{lemma}[proposition]{Lemma}
\newtheorem{corollary}[proposition]{Corollary}

\theoremstyle{definition}
\newtheorem{definition}[proposition]{Definition}
\newtheorem{assumption}[proposition]{Assumption}
\newtheorem{remark}[proposition]{Remark}
\newtheorem{OP}[proposition]{Open Problem}
\numberwithin{equation}{section}
\numberwithin{proposition}{section}
\title[SPDEs with $L^2$ initial data]{On the local well-posedness of randomly forced
	reaction-diffusion equations with 
	$\bm{L^2}$ initial data and a superlinear reaction term}
\thanks{Research supported by the Leverhulme Trust Fellowship IF-2025-040, the US-NSF grants 
	DMS-1855439 and DMS-2245242,
	the Spanish MINECO grant PID2022-138268NB-100, and
	Ayudas Fundacion BBVA a Proyectos de Investigaci\'on Cient\'ifica 2021}
\author[M. Foondun]{Mohammud Foondun}
\address{University of Strathclyde}
\email{mohammud.foondun@strath.ac.uk}
\author[D. Khoshnevisan]{Davar Khoshnevisan}
\address{The University of Utah}
\email{davar@math.utah.edu}
\author[E. Nualart]{Eulalia Nualart}
\address{Universitat Pompeu Fabra and Barcelona School of Economics}
\email{eulalia.nualart@upf.edu}
\date{September 30, 2025}
\let\oldtocsection=\tocsection
\let\oldtocsubsection=\tocsubsection

\renewcommand{\tocsection}[2]{\hspace{0em}\oldtocsection{#1}{#2}}
\renewcommand{\tocsubsection}[2]{\hspace{2.5em}\oldtocsubsection{#1}{#2}}
\begin{document}

\begin{abstract}
	We consider a parabolic stochastic partial differential equation (SPDE) on $[0\,,1]$
	that is forced with multiplicative space-time white noise with a bounded and Lipschitz diffusion 
	coefficient and a  drift coefficient that is locally Lipschitz and
	satisfies an $L\log L$ growth condition. We prove that 
	the SPDE is well posed when the initial data 
	is in $L^2[0\,,1]$. This solves a strong form of an open problem.
\end{abstract}
\maketitle
\section{Introduction}

Consider the stochastic partial
differential equation (SPDE),
\begin{equation}\label{SHE:1}
	\partial_t u(t\,,x) = \tfrac12 \partial^2_x u(t\,,x) + b(u(t\,,x)) + \sigma(u(t\,,x)) \dot{W}(t\,,x),
\end{equation}
where $(t\,,x)\in(0\,,\infty)\times (0\,,1)$, subject to $u(0\,,x) = u_0(x)$,
and with the following Dirichlet boundary conditions on $[0\,,1]$:
\[
	u(t\,,0)=u(t\,,1)=0\qquad\forall t>0.
\]
Throughout, the forcing term $\dot{W}$ is space-time white noise;
that is, $\dot{W}$ is a \textcolor{black}{generalized Gaussian random field with mean zero and
$$\Cov [ \dW(t\,,x) \,, \dW(s\,,y) ] = \delta_0(t-s) \delta_0(x-y),$$
for all $t,s\ge0$ and $x,y\in [0,1]$}.
Additionally, we impose the following assumptions on 
the coefficients $b$ and $\sigma$ in \eqref{SHE:1}:
\begin{assumption}\label{cond-dif}
	The function  $\sigma:\R\rightarrow \R$ is globally
	Lipschitz continuous and  bounded.
	The function $b:\R \rightarrow \R$  is locally
	Lipschitz continuous 
	and satisfies $|b(z)| = \mathcal{O}( |z| \log|z|)$
	as $|z|\to\infty$.
\end{assumption}
Let us suppose in addition to Assumption \ref{cond-dif}
that $u_0\in L^2[0\,,1]$.
Under these conditions, the long-time 
well-posedness of \eqref{SHE:1} is an open problem in the folklore 
of SPDEs. This is due to the fact that for 
$L^2[0\,,1]$ initial data, the solution can become unbounded near time zero, rendering classical truncation and approximation techniques ineffective. Our approach is new and could in fact be applied to a wider class of equations with $L^2[0\,,1]$ initial data. Dalang, Khoshnevisan, and Zhang  \cite{DKZ}*{Theorem 1.4} have 
made progress in this direction by showing that if \eqref{SHE:1} admits
a local-in-time solution $u$, then $u$ can be extended to be a global solution.
To be more precise, they introduce the following.

\begin{definition}[Def.~1.3 of \cite{DKZ}]\label{def:L^2_loc}
	We say that a \eqref{SHE:1} has an \emph{$\mathbb{L}^2_{\textit{loc}}$-solution}
	$u$ if there exists a stopping time $\tau$, with respect to the standard Brownian
	filtration\footnote{This is the filtration generated
	by all processes of the form $t\mapsto \int_{\sD(t)}\phi(x)\,W(\d s\,\d x)$
	as $\phi$ roams over $L^2[0\,,1]$.}  generated by $\dot{W}$, and an adapted continuous
	$L^2[0\,,1]$-valued random field $\{u(t)\}_{t\in[0,\tau)}$ such that
	\begin{align*}
		&\int_0^1 u(t\,,x)\phi(x)\,\d x\\
		&=\int_0^1 u_0(x)\phi(x)\,\d x + \frac12\int_{\sD(t)}u(s\,,x)
			\phi''(x)\,\d s\,\d x + \int_{\sD(t)} b(u(s\,,x))
			\phi(x)\,\d s\,\d x\\
		&\qquad + \int_{\sD(t)} \sigma(u(s\,,x))\phi(x)\,W(\d s\,\d x)
			\hskip1in\text{almost surely on $\{\tau>t\}$,}
	\end{align*}
	for every nonrandom $t>0$
	and all nonrandom test functions $\phi\in C^2[0\,,1]$ that satisfy
	$\phi(0)=\phi(1)=0$, where
	\[
		\sD(t) = (0\,,t]\times[0\,,1].
	\]
\end{definition}
With this definition in place,
Dalang et al (\emph{loc.~cit.}) proved conditionally that if \eqref{SHE:1}
admits an $\mathbb{L}^2_{\textit{loc}}$-solution
$u$, then the maximal time $\tau$ up to which $u$ can be constructed -- that is,
$\sup_{t\in[0,\tau)}\|u(t)\|_{L^2[0,1]}=\infty$ a.s.~-- 
\emph{a priori} satisfies $\P\{\tau=\infty\}=1$.
The existence of $\mathbb{L}^2_{\textit{loc}}$-solutions has been conjectured in
\cite{DKZ} and has since remained open. Shang and Zhang \cite{SZ} make
progress toward this
problem by verifying the conjecture in the case where $\dot{W}$ is replaced by 
temporal white noise. The main goal
of the present article is to resolve a strong form of this conjecture in its originally
stated form.

Let us recall from Dalang \cite{Dalang} that a \emph{random-field solution} 
to \eqref{SHE:1} is a predictable 
random field  $u=\{u(t\,,x)\}_{t \geq 0, x \in [0,1]}$ that satisfies the following integral equation:
\begin{equation}\label{mild_SHE:1}
	u(t\,,x) = ( \cG_tu_0)(x) + I_b(t\,,x) + J_\sigma(t\,,x)\quad\text{a.s.,}
\end{equation}
where $\{\cG_t\}_{t\ge0}$ denotes the heat semigroup, that is,
$(\cG_0 f)(x)=f(x)$ and
\begin{equation}\label{G_tf}
	(\cG_0 f)(x)=f(x),\quad
	( \cG_t f)(x) =\int_0^1 G_{t}(x\,,y) f(y) \, \d y\qquad \forall t>0, x\in[0\,,1],
\end{equation}
for every Lebesgue measurable $f:[0\,,1]\to\R$ 
for which the preceding is a well-defined
Lebesgue integral, and
\begin{equation}\label{G}\begin{split}
	G_t(x\,,y) & = 2\sum_{n=1}^\infty\sin(n\pi x)\sin(n\pi y)\exp( - n^2\pi^2 t/2),\\
	I_b(t\,,x) & = \int_{(0,t)\times [0,1]} G_{t-s}(x\,,y)b(u(s\,,y))\,\d s\,\d y,\\
	J_\sigma(t\,,x)  &= \int_{(0,t)\times [0,1]} G_{t-s}(x\,,y)\sigma(u(s\,,y))\,W(\d s\,\d y).
\end{split}\end{equation}
A part of the definition of a ``mild solution'' tacitly includes the statement that
$I_b$ and $J_\sigma$ are well-defined integrals, respectively in the sense of Lebesgue and Walsh.

It is well known that \eqref{SHE:1} is well posed when $b$ and $\sigma$ 
are Lipschitz continuous and $u_0$ is continuous; see for example
Dalang \cite{Dalang}, Walsh \cite{Walsh}, and Dalang and Sanz-Sol\'e \cite{bookDS}. 
It is also a well-known consequence of a stochastic Fubini argument
that if \eqref{SHE:1} has a continuous random-field solution $u$
on $(0\,,\infty)\times(0\,,1)$, then 
$u$ is in particular an $\mathbb{L}^2_{\textit{loc}}$-solution.
The following is the main contribution of this paper.

\begin{theorem}\label{th:main}
	Suppose that $u_0\in L^2[0\,,1]$,
	and that Assumption \ref{cond-dif} holds. Then, there exists 
	a nonrandom number $t_0>0$ such that 
	\eqref{SHE:1} has a random-field solution $u=u(t\,,x)$ 
	for all $(t\,,x)\in(0\,,t_0]\times[0\,,1]$ that
	satisfies the following:
	\begin{compactenum}[\indent\indent\rm (1)]
	\item $(t\,,x)\mapsto u(t\,,x)$ is a.s.\ continuous on $(0\,,t_0]\times[0\,,1]$;
	\item $\lim_{t\to0+} t^\alpha \|u(t)\|_{C[0,1]}=0$ a.s.\ 
		for every $\alpha>\frac14$, in fact there exists $\gamma=\gamma(b\,,\sigma,\alpha) >0$ such that
		\[
			\E\exp\left\{\gamma\sup_{t\in(0,t_0]}(t^\alpha\|u(t)\|_{C[0,1]})^{2/3}
			\right\}<\infty;
		\]
	\item As $t\to0+$, $\| u(t)-u_0\|_{L^2[0,1]}\to 0$ in probability.
	\end{compactenum}
	Finally, if $v$ is any other continuous random-field solution to \eqref{SHE:1} on 
	the time interval $(0\,,t_0]$ such that 
	$\sup_{s\in(0,t_0]}(s^\alpha\|v(s)\|_{C[0,1]})<\infty$ a.s.~for 
	some $\alpha>\frac14$, then $\P\{u(t)=v(t)\ \forall t\in(0\,,t_0]\}=1$.
\end{theorem}

As was mentioned earlier,
Dalang et al \cite{DKZ}  proved conditionally in their Theorem 4.1 that, under 
Assumption \ref{cond-dif} and when $u_0\in L^2[0\,,1]$, if 
\eqref{SHE:1} admits a local solution $u$, then $u$ extends to a long-time solution. This reduced the 
conjectured well-posedness of
\eqref{SHE:1} to the existence of local solutions. Theorem \ref{th:main} above resolves precisely that problem
and implies that \eqref{SHE:1} is well posed, for all time, under Assumption \ref{cond-dif}.

Before we discuss some of the history of the problem that led to Theorem \ref{th:main},
let us mention two questions that have eluded us.

\begin{OP}
	Does Theorem \ref{th:main} continue to hold if the $L^2$-condition on the initial profile
	is replaced by $u_0\in L^p[0\,,1]$ for some $p<2$? We suspect this might
	be true when $p\in(1\,,2)$.
\end{OP}

\begin{OP}
	We believe the method of proof of Theorem \ref{th:main} ought to allow one to prove that the same
	result holds more generally when the nonrandom function
	$\sigma$ is Lipschitz continuous and satisfies 
	$|\sigma(z)| = \mathcal{O}(|z|^\eta)$,
	as $|z|\to\infty$, for some $\eta\in(0\,,1)$. Is Theorem \ref{th:main} true 
	when $\eta=1$ -- that is, for
	every nonrandom, Lipschitz continuous $\sigma$?
\end{OP}

There is a large literature that studies the well-posedness of  \eqref{SHE:1}
in the case that $u_0$ is non-negative and continuous and $(b\,,\sigma)$
are not necessarily globally Lipschitz continuous.
When $\sigma$ is a positive constant and $b$ is a locally Lipschitz function, 
Bonder and Groisman \cite{BonderGroisman} proved that the solution to 
\eqref{SHE:1} blows up in finite time whenever $b$ is nonnegative, convex, 
and  satisfies the celebrated \emph{Osgood condition}, $\int_1^{\infty} b(x)^{-1}\,\d x < \infty$,
where $0^{-1}=\infty$. In their paper ({\it loc.~cit.}), Dalang, Khoshnevisan, and Zhang
also investigate the optimality of the Osgood condition by proving that if $b$ satisfies its portion of
Assumption \ref{cond-dif} and if $\sigma$ is locally Lipschitz and satisfies 
$|\sigma(z)| = \mathcal{O}( |z| (\log|z|)^{1/4})$
as $|z|\to\infty$, then there exists a global continuous solution to \eqref{SHE:1}.
Foondun and Nualart \cite{FN2021} proved that the Osgood condition is also necessary 
when $\sigma$ is locally Lipschitz and bounded. And they
extended their  results to the stochastic wave  equation in \cite{FN2022}.  Non-explosion for the solution to \eqref{SHE:1} for super-linear $\sigma$ and $b$ are studied more recently by Salins \cite{Salins2021, Salins2022}, and by Foondun, Khoshnevisan, and Nualart \cite{FKN2} in the case of the real line with $b$ and $\sigma$ locally Lipschitz but with at most linear growth. We refer the reader to these papers for the latest results as well as bibliography.

In a somewhat different direction, the well-posedness of \eqref{SHE:1} was studied 
earlier by  Mueller \cite{Mueller1991} 
in the case that $b=0$ and $\sigma(u)=
u^{\gamma}$, when $\gamma\in(1\,,\frac32)$; see
\cite{Mueller98,HanKim,HanYi,Krylov} for related results.
Subsequently,
Mueller and  Sowers \cite{MuellerSowers} and Mueller \cite{Mueller2000} showed that the solution 
to \eqref{SHE:1} blows up with positive probability when $\gamma>\frac32$. 
In this direction, Salins \cite{Salins2025} has recently resolved the long-standing open problem of
what happens at criticality by establishing non-explosion in the 
critical regime $\gamma=\frac32$.

Blow-up questions for \eqref{SHE:1} on the real line have  also been explored when the initial profile $u_0$
is a bounded function. For instance,
Foondun and Parshad \cite{FP} established finite-time blowup when the  
initial condition is positive and bounded  away from zero, and $\sigma$ and $b$ have
grow faster than linearly  at  infinity.
More recently, Khoshnevisan, Foondun, and Nualart \cite{FKN} proved 
the solution to \eqref{SHE:1} blows up instantaneously, and  everywhere,
when $\sigma$ is globally Lipschitz and bounded away from zero 
and infinity, and $b$ is nonnegative, nondecreasing, locally Lipschitz, and satisfies the  Osgood condition.

For well-posedness results on the Hilbert-space 
approach of equation \eqref{SHE:1} we refer to
da Prato and Zabczyk \cite{dPZ} and Cerrai \cite{Cerrai}.
Well-posedness and related problems for superlinear 
$\sigma$ and/or $b$ with initial data  $u_0\in L^2[0\,,1]$ have been
studied recently by Pan, Shang, and Zhang \cite{PSZ},
Shang and Zhang \cite{SZ}, and Li, Shang, and Zhai \cite{LSZ}.
Finally, we add that the blowup phenomenon for nonrandom PDEs is 
a huge literature on its own, wherein the absence of noise allows for different phenomena altogether;
see for example the comprehensive treatment of Quittner and Souplet \cite{QuittnerSouplet}.

The main innovation of this paper involves the introduction of
a new truncation method that takes into account the possibility that
the solution to \eqref{SHE:1} can be badly unbounded at time zero.
As opposed to better-understood cut-off techniques such as those in
Dalang et al (\emph{loc.~cit.}), 
Salins \cite{Salins2021,Salins2022}, Chen and Huang \cite{ChenHuang},
and Chen, Foondun, Huang, and Salins \cite{CDHS}, we are
motivated by approximation ideas in Miao and Yuan 
\cite{MiaoYuan}. Though we hasten to add that our truncation is significantly different in effect
from those of Miao and Yuan ({\it loc.~cit.})
as we work in physical, rather than Fourier, space.
In this way, we are led to \emph{a priori} approximations that use stopping times
and associated  estimates that control the behavior of 
those badly behaved local solutions.
The better-established truncation arguments in the literature -- such as those
in \cite{DKZ,Salins2021,Salins2022,ChenHuang,CDHS} -- can fail 
in our context primarily because
general $L^2$-initial data can \emph{a priori} lead to a solution 
that can become very large essentially instantaneously. 

Let us end the Introduction with a brief outline of the paper. 
In \S2 we introduce certain Banach spaces of locally bounded functions;
these spaces are crucial
to our subsequent analysis.
Sections 3-5 are  dedicated to various inequalities and \emph{a priori} estimates
for bounded  random fields. In \S6 we prove  the existence, uniqueness, regularity, and
stability of a  generalization of the SPDE \eqref{SHE:1} in which the coefficients depend also on time.
The remaining details of the proof of Theorem \ref{th:main} are
gathered  in \S7, and use the earlier results of the paper, including those in \S6 about
SPDEs with temporally dependent coefficients.

Throughout this paper, we write 
\textcolor{black}{$\|X\|_p = \E(|X|^p)^{1/p}$
for all $p\ge1$ and $X\in L^p(\Omega)$.}
For every space-time function $f:(0\,,\infty)\times\R\to\R$, 
$\lip(f)$ denotes the optimal Lipschitz constant of the spatial variable of $f$
uniformly in time; that is,
\begin{equation*}
	\lip(f) =  \sup_{t>0}\sup_{a,b\in\R: a\neq b}
	\frac{|f(t\,,b)-f(t\,,a)|}{|b-a|}.
\end{equation*}
If $f$ depends only on the spatial variable $x$, then $\lip(f)$ still makes sense, since we can 
(and will) tacitly extend $f$
to a space-time function as follows $f(t\,,x) = f(x)$.
Thus, for fixed $t>0$, we write $\lip(f(t))$ for the Lipschitz constant of the map $x \mapsto f(t\,,x)$.
 We denote by $\lip(\mathbb{X})$ the collection
of real-valued, globally Lipschitz functions on any subset $\mathbb{X}$ of a Euclidean space.

Throughout this paper, we extend the last part of Definition \ref{def:L^2_loc} by setting
\begin{equation}\label{D}
	\sD(t) = (0\,,t]\times[0\,,1] \quad \forall  t>0,
	\qquad
	\sD(\infty)=(0\,,\infty)\times[0\,,1].
\end{equation}
Moreover, $\log$ denotes the natural logarithm, and $\log_+(a)=\log(a+\e)$ for all $a\ge0$.

\section{Banach spaces of locally bounded functions}
For all $T\in(0\,,\infty]$ and $\alpha\ge0$, $\beta>0$, let $\sC_T(\alpha\,,\beta)$ denote the collection of all 
continuous functions
$f:\sD(\infty)\to\R$ such that $\|f\|_{\sC_T(\alpha,\beta)}<\infty$, where
\begin{equation}\label{sec:Banach}
	\|f\|_{\sC_T(\alpha,\beta)}=
	\sup_{(t,x)\in \sD(T)}
	\left( t^\alpha \e^{-\beta t} |f(t\,,x)| \right),
\end{equation}
and
\begin{equation}\label{N}
	\cN_{k,\alpha,\beta,T}(X) = \sup_{(t,x)\in \sD(T)}
	\left( t^\alpha \e^{-\beta t} \|X(t\,,x)\|_k\right).
\end{equation}
The norms $\cN_{k,0,\beta,\infty}$ appeared
first in Foondun and Khoshnevisan \cite{FK} in order to analyse the intermittency properties
of various families of parabolic SPDEs. The norms in \eqref{N} are different from those in 
\cite{FK} in two ways:
\begin{compactenum}[(i)]
\item For the bulk of our purposes, $T$ is finite
	here; this change is not a major difference and would have worked
	equally well  in the work of \cite{FK};
\item In contrast with the norms in \cite{FK} where $\alpha=0$, 
	we will be only interested in cases where $\alpha\in(0\,,1)$. This
	creates for a totally different behavior of the norms $\cN_{k,\alpha,\beta,T}$.
	Whereas in \cite{FK} $[\alpha=0]$, the norms measure the large-$t$ behavior
	of the random field, here $[\alpha>0]$, the norms gauge the small-$t$
	behavior of $X$. In fact, the critical value of the optimized $t$ 
	in \eqref{N} turns out to be $1/\beta$ 
	when $\beta\gg1$ -- this is the case of interest.
\end{compactenum}

Next we elaborate further on Item (ii) above.
Before we present the appropriate result
let us observe that, because of the obvious inequality,
\[
	[ \cN_{k,\alpha,\beta,T}(X) ]^k \le \E(\|X\|_{\sC_T(\alpha,\beta)}^k),
\]
the quantity $\cN_{k,\alpha,\beta,T}(X)$ is finite
whenever $\E(\|X\|_{\sC_T(\alpha,\beta)}^k)$ is finite, and
this is valid for all $\alpha,\beta>0$ and $0<T\le\infty$.
The following provides a kind of quantitative converse to this.
	
\begin{proposition}\label{pr:norm<N}
	Choose and fix real numbers $\alpha>0$, $\beta\ge1$, $\tau,\mu\in(0\,,1)$, and 
	let $\{X(t\,,x)\}_{(t,x)\in \sD(\infty)}$ denote a space-time random field
	such that, for some $k \ge 4(\tau^{-1}+\mu^{-1})$,
	\begin{equation}\label{Kolm:cond}
		C = \adjustlimits\sup_{\varepsilon\in(0,1)}\sup_{t>0}\sup_{0\le x<y\le1}\left(
		\frac{t^\alpha\e^{- \beta t}\|X(t+\varepsilon\,,x) - X(t\,,y)\|_k}{%
		\varepsilon^\tau \vee |x-y|^\mu}\right)<\infty.
	\end{equation}
	Then, $\E( \|X\|_{\sC_\infty(\bar\alpha,\bar\beta)}^k)
	\le (1280L)^k[ \cN_{k,\alpha,\beta,\infty}(X)+C]^k$
	for all $\bar\alpha>\alpha$ and $\bar\beta > \beta$,
	where
	\begin{equation}\label{L}\begin{split}	
		L = L(\alpha\,,\bar\alpha\,,\beta\,,\bar\beta) =
		\max\left( \frac{2^{\bar\alpha}}{2^{\bar\alpha-\alpha} - 1}
		~, \sum_{n=0}^\infty (n+1)^{\bar\alpha+\alpha}\
		\e^{-n(\bar\beta-\beta)/\beta}\right).
	\end{split}\end{equation}
\end{proposition}

It is clear from the forthcoming proof of Proposition \ref{pr:norm<N}
that our choice of $L$ in \eqref{L} is far from optimal. We mention 
the expression \eqref{L} only to make clear the assertion that, except
for its dependence on $(\alpha\,,\bar\alpha\,,\beta\,,\bar\beta)$,
the constant $L$ can be selected in a universal fashion. In particular,
we note that $L$ depends only on $(\bar\alpha\,,\alpha)$, and not on
$\beta$,
when $\bar\beta=2\beta$ (say). We will use the latter property of $L$
in a critical way later on.

\begin{remark}\label{rem:norm<N}
	Proposition \ref{pr:norm<N} readily implies a local-in-time version.
	Indeed, we can replace $X(t)$ by $X(t\wedge T)$ for a fixed value
	of $T\in(0\,,\infty)$ in order to see that if
	there exist real numbers $k \ge 4(\tau^{-1}+\mu^{-1})$
	and $C>0$ such that
	\begin{equation*}
		C_T = C_T(\alpha\,,\beta\,,\tau\,,\mu) 
		= \sup_{0<s<t\le T}\sup_{0\le x<y<1}\left(\frac{
		s^\alpha\e^{- \beta s}\|X(s\,,x) - X(t\,,y)\|_k}{ |t-s|^\tau \vee |x-y|^\mu}
		\right)<\infty,
	\end{equation*}
	then, for the same number
	$L$ as in \eqref{L}, and for all $\bar\alpha>\alpha$ and $\bar\beta > \beta$, 
	\[
		\E\left( \|X\|_{\sC_T(\bar\alpha,\bar\beta)}^k\right)
		\le (1280L)^k[ \cN_{k,\alpha,\beta,T}(X)+C_T]^k.
	\]
\end{remark}

The proof of  Proposition \ref{pr:norm<N} hinges on the following
formulation of Garsia's lemma \cite{G}
(also known under the Kolmogorov continuity theorem, 
Garsia-Rodemich-Rumsey inequality, Kolmogorov-\c{C}entsov theorem,
quantitative chaining, \dots) for 2-parameter processes.

\begin{lemma}\label{lem:Garsia}
	For every pairs of real numbers $\tau,\mu\in(0\,,1)$, 
	and for all closed intervals $I\subset(0\,,\infty)$ of length $\le1$, let us
	define
	\begin{equation}\label{rho}
		\varrho( (s\,,y)\,,(t\,,x) ) = |t-s|^{\tau} \vee |x-y|^\mu
		\quad\text{for all $(t\,,x), (s\,,y)\in I\times[0\,,1]$}.
	\end{equation}
	Then, for every $f\in C( I\times[0\,,1])$ and $k\ge 4(\tau^{-1}+\mu^{-1})$,
	\[\displaystyle
		\sup_{y\in I\times[0,1]}|f(y)| - \inf_{y\in I\times[0,1]}|f(y)| \le 640
		\left(\int_{I\times[0,1]}\d a\int_{I\times[0,1]}\d a'\
		\left| \frac{f(a)-f(a')}{\varrho(a\,,a')}\right|^k\right)^{1/k}.
	\]
\end{lemma}

Thanks to the proof of the
Kolmogorov continuity theorem for two-parameter processes, the optimal condition for boundedness,
in the above context, is $k>2(\tau^{-1}+\mu^{-1})$; see also the proof below. As we will see,
the slightly stronger condition  $k\ge 4(\tau^{-1}+\mu^{-1})$  produces a universal
constant (here, 640), among other things independently of $(k\,,\mu\,,\tau)$.

\begin{proof}
	We apply the particular form of Garsia's lemma \cite{G}
	from Dalang, Khoshnevisan, and Nualart 
	\cite[Proposition A.1]{DKN} with 
	$S=\R^2$, $K=I\times[0\,,1]$, $\varrho$ given by \eqref{rho}, $\nu=$ the Lebesgue measure on
	$\R^2$, $\Psi(x)=|x|^k$, and $p(u)=u$ in order to see that for all real numbers $k\ge1$,
	\begin{align}\nonumber
		&\sup_{y\in I\times[0,1]}|f(y)| - \inf_{y\in I\times[0,1]}|f(y)| \\
		&\le\sup\left\{ |f(a)-f(a')| :\, a,a'\in I\times[0\,,1],\ \varrho(a\,,a')\le1\right\}
			\label{Garsia}\\\nonumber
		&\le 10 \left(\int_{I\times[0,1]}\d a\int_{I\times[0,1]}\d a'\
			\left| \frac{f(a)-f(a')}{\varrho(a\,,a')}\right|^k\right)^{1/k}
			\sup_{w\in I\times[0,1]}\int_0^2
		\frac{\d u}{| B_\varrho(w\,,u/4)|^{2/k}},
	\end{align}
	where the second $|B_\varrho(w\,,u/4)|$ denotes the Lebesgue measure of $B_\varrho(w\,,u/4)$, and
	$B_\varrho(w\,,r) = \{v\in I\times[0\,,1]:\,\varrho(w\,,v) < r\}
	= ( t -  r^{1/\tau} , t +  r^{1/\tau} ) \times
	( x -  r^{1/\mu}, x +  r^{1/\mu} ),$
	for all $w\in I\times[0\,,1]$ and $r>0$.
	Let $\zeta = \tau^{-1}+\mu^{-1}$, and conclude from above that
	\[
		\int_0^2| B_\varrho(w\,,u/4)|^{-2/k}\,\d u < \int_0^2
		(4/u)^{2\zeta/k}\,\d u \le 64,
	\]
	uniformly for all $w\in I\times[0,1]$, $k\ge 4\zeta$ and closed
	intervals $I\subset(0\,,\infty)$ of length $\le1$.
	This and \eqref{Garsia} together imply the lemma.
\end{proof}

\begin{proof}[Proof of Proposition \ref{pr:norm<N}]
	Let $I(n) = [2^{-n-1}\beta^{-1},2^{-n}\beta^{-1}]$ for all $n\in\Z_+$.
	Recall \eqref{rho} and apply \eqref{Kolm:cond} to see that
	\begin{equation}\label{intintI(n)}\begin{split}
		&\int_{I(n)\times[0,1]}\d t\,\d x
			\int_{I(n)\times[0,1]}\d s\,\d y\
			\E\left(\left| \frac{X(t\,,x)-X(s\,,y)}{\varrho((t\,,x)\,,(s\,,y))}\right|^k\right)\\
		&\hskip3in
			\le \frac{C^k \exp\left( 2^{-n}k/\beta \right)}{2^{(n+1)\alpha k}\beta^{\alpha k}}.
	\end{split}\end{equation}
	The length of $I(n)$ is $\le 1$ since $\beta\ge1$.
	Therefore, Lemma \ref{lem:Garsia} implies that
	\begin{align*}
		\left\| \sup_{I(n)\times[0,1]}|X| \right\|_k - \left\| \inf_{I(n)\times[0,1]}|X| \right\|_k
		&\le 640\left[\frac{C^k \exp\left( 2^{-n}k/\beta \right)}{2^{(n+1)\alpha k}\beta^{\alpha k}}\right]^{1/k}\\
		&\le 640C\e 2^{-n\alpha}\beta^{-\alpha},
	\end{align*}
	with room to spare. In particular, whenever $\bar\alpha>\alpha>0$ and $\bar\beta>\beta\ge1$,
	\begin{align*}
		&\left\| \|X\|_{\sC_{1/\beta}(\bar\alpha,\bar\beta)}\right\|_k
			=\left\| \sup_{(t,x)\in(0,1/\beta]\times[0,1]}
			\left( t^{\bar\alpha} \e^{-\bar\beta t} |X(t\,,x)| \right)
			\right\|_k \\
		&\le \sum_{n=0}^\infty
			\left\| \sup_{(t,x)\in I(n)\times[0,1]} \left( t^{\bar\alpha}  |X(t\,,x)| \right)
			\right\|_k 	\le \beta^{-\bar\alpha}\sum_{n=0}^\infty 2^{-n\bar\alpha}
			\left\| \sup_{I(n)\times[0,1]} |X| \right\|_k\\
		&\le \beta^{-\bar\alpha}\sum_{n=0}^\infty 2^{-n\bar\alpha}
			\inf_{(t,x)\in I(n)\times[0,1]}\left\|  X(t\,,x)\right\|_k
			+640C\e \beta^{-\bar\alpha-\alpha}\sum_{n=0}^\infty 2^{-n(\alpha+\bar\alpha)}.
	\end{align*}
	Because $\sup_{(t,x)\in I(n)\times[0,1]}\|  X(t\,,x)\|_k$ is bounded from above by
	\[
		2^{(n+1)\alpha}\beta^\alpha\exp(2^{-n})\cN_{k,\alpha,\beta,1/\beta}(X)
		\le 2^{(n+1)\alpha}\beta^\alpha\e\cN_{k,\alpha,\beta,\infty}(X),
	\]
	it follows  from the preceding that
	\begin{equation}\label{first}\begin{split}
		\left\| \|X\|_{\sC_{1/\beta}(\bar\alpha,\bar\beta)}\right\|_k
			&\le \frac{2^\alpha\e \beta^{-(\bar\alpha-\alpha)}
			}{1-2^{-(\bar\alpha-\alpha)}}\,\cN_{k,\alpha,\beta,\infty}(X)
			+ \frac{640C\e \beta^{-(\bar\alpha+\alpha)}}{1-2^{-(\bar\alpha+\alpha)}}\\
		&\le 640L\left[ \cN_{k,\alpha,\beta,\infty}(X)+ C\right],
	\end{split}\end{equation}
	where $L$ is defined in \eqref{L}.
	
	Next, let $J(n) = [n/\beta \,,(n+1)/\beta]$ for all $n\in\N$, and deduce from \eqref{Kolm:cond} that
	\[
		\int_{J(n)\times[0,1]}\d t\,\d x
		\int_{J(n)\times[0,1]}\d s\,\d y\
		\E\left(\left| \frac{X(t\,,x)-X(s\,,y)}{\varrho((t\,,x)\,,(s\,,y))}\right|^k\right)
		\le \frac{C^k\beta^{\alpha k}\e^{(n+1)k}}{n^{\alpha k}}.
	\]
	[Compare with \eqref{intintI(n)}.] 
	The length of $J(n)$ is $\le 1$ since $\beta\ge1$.
	Therefore, Lemma \ref{lem:Garsia} implies that
	\[
		\left\| \sup_{J(n)\times[0,1]}|X| \right\|_k - \left\| \inf_{J(n)\times[0,1]}|X| \right\|_k
		\le 640C\beta^\alpha\e^{n+1} n^{-\alpha} < 640C\beta^\alpha\e^{n+1},
	\]
	with room to spare. In particular, whenever $\bar\alpha>\alpha>0$ and $\bar\beta>\beta\ge1$,
	\begin{align*}
		&\left\| \sup_{(t,x)\in [1/\beta,\infty)\times[0,1]}  
			\left( t^{\bar\alpha} \e^{-\bar\beta t} |X(t\,,x)| \right)
			\right\|_k \le \sum_{n=1}^\infty
			\left\| \sup_{(t,x)\in J(n)\times[0,1]} 
			\left( t^{\bar\alpha}\e^{-\bar\beta t}  |X(t\,,x)| \right)
			\right\|_k \\
		&\le \beta^{-\bar\alpha}{\sum_{n=0}^\infty} (n+1)^{\bar\alpha}\
			\e^{-\bar\beta n/\beta} 
			\left\| \sup_{J(n)\times[0,1]} |X| \right\|_k\\
		&\le \beta^{-\bar\alpha}\sum_{n=0}^\infty (n+1)^{\bar\alpha}\
			\e^{-\bar\beta n/\beta} 
			\inf_{(t,x)\in J(n)\times[0,1]}\left\|  X(t\,,x)\right\|_k
			+ 640C\beta^{-(\bar\alpha-\alpha)}\sum_{n=0}^\infty
			\e^{-n(\bar\beta-\beta)/\beta}\\
		&\le\sum_{n=0}^\infty (n+1)^{\bar\alpha}\
			\e^{-\bar\beta n/\beta} 
			\inf_{(t,x)\in J(n)\times[0,1]}\left\|  X(t\,,x)\right\|_k
			+ \frac{640C}{1-\exp\{ -(\bar\beta-\beta)/\beta\} },
	\end{align*}
	with room to spare.
	Because 
	\[
		\inf_{(t,x)\in J(n)\times[0,1]}\left\|  X(t\,,x)\right\|_k
		\le (n/\beta)^\alpha\e^n
		\cN_{k,\alpha,\beta,n/\beta}(X)
		\le n^\alpha\e^n \cN_{k,\alpha,\beta,\infty}(X),
	\]
	it follows  from the preceding that
	\begin{align*}
		&\left\| \sup_{(t,x)\in [1/\beta,\infty)\times[0,1]}  
			\left( t^{\bar\alpha} \e^{-\bar\beta t} |X(t\,,x)| \right)
			\right\|_k \\
		&\le \cN_{k,\alpha,\beta,\infty}(X)
			\sum_{n=0}^\infty \frac{(n+1)^{\bar\alpha+\alpha}}{
			\exp\{n(\bar\beta-\beta)/\beta\}}
			+ \frac{640C}{1-\exp\{ -(\bar\beta-\beta)/\beta\} }\\
		&\le 640L\left[ \cN_{k,\alpha,\beta,\infty}(X)+C\right];
	\end{align*}
	see \eqref{L}.
	Add this to \eqref{first} in order to deduce the bulk of the proposition;
	that is, 
	\begin{equation}\label{mom:bd}
		\E( \|X\|_{\sC_\infty(\bar\alpha,\bar\beta)}^k)
		\le (1280L)^k[ \cN_{k,\alpha,\beta,\infty}(X)+C]^k.
	\end{equation}
	In order to complete the proof it remains to verify
	that $X\in\sC_T(\alpha\,,\beta)$ a.s.\ for all $T\in(0\,,\infty)$. Because
	of the already-proved moment bound \eqref{mom:bd}, we need only check
	that $X$ is a.s.\ continuous on $\sD(\infty)$. But that follows
	immediately from \eqref{Kolm:cond} and a suitable form of
	a 2-parameter Kolmogorov continuity theorem
	such as the previously-mentioned Proposition A.1 of \cite{minicourse}.
\end{proof}

\section{Some preliminary integral inequalities}

Recall the Green function defined in (\ref{G}). In this section we collect
some elementary real-variable properties of that Green function. The first
is a small variation of a very well-known property of all such Green functions.
We include the short proof for the sake of completeness.

\begin{lemma}\label{lem:G}
	$\int_0^1 G_s(x\,,y)\,\d y\le1$ and 
	$0\le G_s(x\,,z) \le 1/\sqrt s$ for all $s>0$ and $x,z\in[0\,,1]$.
\end{lemma}

\begin{proof}
	$(0\,,\infty)\times[0\,,1]^2\ni
	(s\,,x\,,y)\mapsto G_s(x\,,y)$ is also the transition probability density for a Brownian motion,
	run at twice the standard speed, and killed when it reaches $\{0\,,1\}$; see
	Bass \cite[Ch.~2, \S7]{Bass}. This implies that
	$\int_0^1 G_s(x\,,y)\,\d y$ is the probability that the same Brownian motion has not yet been killed
	by time $s$, started at $x\in[0\,,1]$. Because of this,
	\eqref{G} implies that
	\[
		0\le G_s(x\,,z) \le 2\sum_{n=1}^\infty \exp(-n^2\pi^2 s/2)
		\le\int_{-\infty}^\infty\exp(-w^2\pi^2s/2)\,\d w,
	\]
	and the lemma follows.
\end{proof}

The following continuity estimate can be found in Dalang and Sanz-Sol\'e
\cite[Lemma B.2.1]{bookDS}.

\begin{lemma}\label{lem:G-G:x}
	For all $x,z\in[0\,,1]$, 
	\[
		\int_0^\infty\d s\int_0^1 |G_s(x\,,y)-G_s(z\,,y)|^2\,\d y \lesssim |x-z|,
	\]
	where the implied constant is universal.
\end{lemma}
%
%

%

Next we list a weighted integral inequality for the Green function.

\begin{lemma}\label{lem:G-G:L1:x}
	Choose and fix $\alpha, \delta\in[0\,,1)$, and $\chi\ge0$.
	Then, 
	\[
		t^\alpha\e^{-\beta t}\int_0^t s^{-\alpha}
		|\log_+(1/s)|^\chi\, \e^{\beta s}\,\d s\int_0^1\d y\
		\left| G_{t-s}(x\,,y) - G_{t-s}(z\,,y) \right|
		\lesssim \frac{|x-z|^\delta}{\beta^{1-\delta}},
	\]
	 uniformly for all $t\in(0\,,1]$, $\beta>0$, and $x,z\in[0\,,1]$.
\end{lemma}

\begin{proof}
	Thanks to \eqref{G}, uniformly for all $t>s>0$ and $x,z\in[0\,,1]$,
	\begin{align*}
		\int_0^1 \left| G_{t-s}(x\,,y) - G_{t-s}(z\,,y) \right|\d y 
			&\le2\sum_{n=1}^\infty
			|\sin(n\pi x)-\sin(n\pi z)|\e^{-n^2\pi^2(t-s)/2}\\
		&\le2\pi\sum_{n=1}^\infty (n|x-z|\wedge1)\e^{-n^2\pi^2(t-s)/2}.
	\end{align*}
	This sum can be estimated via an integral test in order to find that
	the preceding is $\lesssim |x-z| / \sqrt{t-s}$
	uniformly for all $t>s>0$ and $x,z\in[0\,,1]$. And the simple estimate 
	$|G_{t-s}(x\,,y)-G_{t-s}(z\,,y)|\le G_{t-s}(x\,,y)+G_{t-s}(z\,,y)$, used in conjunction
	with Lemma \ref{lem:G}, yields the bounds,
	\begin{equation}\label{G:L1:x}\begin{split}
		\int_0^1 \left| G_{t-s}(x\,,y) - G_{t-s}(z\,,y) \right|\d y\lesssim
		\min\left(1\,,\, \frac{|x-z| }{\sqrt{t-s}}\right)
		\le \frac{|x-z|^\delta}{|t-s|^{\delta/2}},
	\end{split}\end{equation}
	valid uniformly for all $t>s>0$, $x,z\in[0\,,1]$, and $\delta\in[0\,,1]$. 
	We apply \eqref{G:L1:x} by splitting the integral $\int_0^t(\,\cdots)$
	of the lemma as $\int_0^{t/2}+ \int_{t/2}^t$.
	
	If $s\in(0\,,t/2)$, then $\sqrt{t-s}\ge\sqrt{t/2}$
	and $\exp(\beta s)\le \exp(\beta t/2)$, whence it follows from
	\eqref{G:L1:x} that for every  fixed choice of $\delta\in(0\,,1]$, and $\chi\in[0\,,1]$,
	\begin{align*}
		& t^\alpha \e^{-\beta t}\int_0^{t/2} s^{-\alpha}|\log_+(1/s)|^\chi\,
			\e^{\beta s}\,\d s \int_0^1\d y\
			\left| G_{t-s}(x\,,y) - G_{t-s}(z\,,y) \right|\\
		&\hskip1.5in\lesssim
			t^{\alpha - \frac\delta2} 
			\e^{-\beta t/2}\int_0^{t/2} s^{-\alpha}|\log_+(1/s)|^\chi\,\d s\
			|x-z|^\delta,
	\end{align*}
	uniformly for all $\beta >0$, $t\in(0\,,1]$, 
	$\chi\ge0$, and $x,z\in[0\,,1]$; the implied constants depend only on 
	$\delta$. Because $|\log(1/s)|^\chi\lesssim
	s^{-\delta/2}$ uniformly for all $s\in(0\,,1]$, this yields
	\begin{equation}\label{(I)}\begin{split}
		& t^\alpha \e^{-\beta t}\int_0^{t/2} s^{-\alpha}|\log_+(1/s)|^\chi\,
			\e^{\beta s}\,\d s \int_0^1\d y\
			\left| G_{t-s}(x\,,y) - G_{t-s}(z\,,y) \right|\\
		&\lesssim t^{1-\delta} 
			\e^{-\beta t/2}|x-z|^\delta \le
			\sup_{r>0}\left[ r^{1-\delta}\e^{-\beta r/2}\right]|x-z|^\delta
			\propto \beta^{-(1-\delta)}  |x-z|^\delta,
	\end{split}\end{equation}
	uniformly for all $\beta>0$, $t\in(0\,,1]$, 
	and $x,z\in[0\,,1]$, and the implied constants depend only on 
	$(\chi\,,\delta)$.
	
	Next we study $\int_{t/2}^t(\,\cdots)$. 
	Note that, similarly to the above case,
	\begin{align*}
		& t^\alpha\e^{-\beta t}\int_{t/2}^t s^{-\alpha}|\log_+(1/s)|^\chi
			\e^{\beta s}\,\d s \int_0^1\d y\
			\left| G_{t-s}(x\,,y) - G_{t-s}(z\,,y) \right|\\
		& t^\alpha\e^{-\beta t}\int_{t/2}^t s^{-\alpha-(\delta/2)}
			\e^{\beta s}\,\d s \int_0^1\d y\
			\left| G_{t-s}(x\,,y) - G_{t-s}(z\,,y) \right|\\
		&\lesssim t^{\delta/2}   |x-z|^\delta 
			\int_{t/2}^t (t-s)^{-\delta/2}
			\e^{-\beta (t-s)} \,\d s
			\le |x-z|^\delta  \int_0^\infty s^{-\delta}\e^{-\beta s}\,\d s
			\propto \frac{|x-z|^\delta}{\beta^{1-\delta}},
	\end{align*}
	using somewhat crude bounds, all valid
	uniformly for all $\beta>0$, $t\in(0\,,1]$.  Combine the above
	with \eqref{(I)} in order to deduce the lemma. 
\end{proof}

The next two lemmas are just real-variable lemmas, and do not involve
the Green function, but are proved using similar methods as in the previous
lemma.

\begin{lemma}\label{lem:elem}
	If $\delta,\alpha\in(0\,,1)$ and $\chi\in[0\,,1]$, then
	\[
		t^\alpha\e^{-\beta t}\int_0^t s^{-\alpha}|\log_+(1/s)|^\chi\, \e^{\beta s}\,\d s
		\lesssim \beta^{-(1-\delta)},
	\]
	uniformly for all $t\in(0\,,1]$ and $\beta>0$. When $\chi=0$, we can choose even $\delta=0$.
\end{lemma}

\begin{proof}
	For every $\delta,\alpha\in(0\,,1)$, $\chi\in[0\,,1]$, $t\in(0\,,1]$, and $\beta>0$,
	\begin{align*}
		&\int_0^t s^{-\alpha}|\log_+(1/s)|^\chi\, \e^{\beta s}\,\d s 
			\lesssim \e^{\beta t/2} \int_0^{t/2} s^{-\alpha-\delta}\,\d s
			+\e^{\beta t} t^{-\alpha-\delta}\int_{t/2}^t \e^{-\beta (t-s)}\,\d s\\
		&\lesssim \e^{\beta t/2} t^{1-\alpha-\delta} +
			\e^{\beta t}t^{-\alpha-\delta}\int_0^{t/2}\e^{-\beta s}\,\d s
			 \lesssim \e^{\beta t/2} t^{1-\alpha-\delta} +
			\e^{\beta t}t^{-\alpha}\int_0^\infty s^{-\delta}\e^{-\beta s}\,\d s\\
		& \propto\e^{\beta t/2} t^{1-\alpha-\delta} +
			\e^{\beta t}t^{-\alpha}\beta^{-(1-\delta)},
	\end{align*}
	where the implied constants only depend on $(\chi\,,\delta)$. Consequently,
	\[
		t^\alpha\e^{-\beta t}\int_0^t s^{-\alpha}|\log_+(1/s)|^\chi\, \e^{\beta s}\,\d s
		\lesssim \sup_{t>0}\left[t^{1-\delta}\e^{-\beta t/2}\right] 
		+\beta^{-(1-\delta)},
	\]
	where the implied constants only depend on $(\chi\,,\delta)$. 
	This completes the proof in the case that $\chi>0$. When $\chi=0$, we go through the preceding
	with $\delta=0$ line by line to see that it continues to be valid. This completes the proof.
\end{proof}

\begin{lemma}\label{lem:rv}
	If $\alpha\in(0\,,1)$ and $\chi\ge0$ are fixed, then 
	\[
		t^\alpha \e^{-\beta t}\int_0^t s^{-\alpha}|\log_+(1/s)|^\chi\,
		\e^{\beta s}\,\d s
		\lesssim \frac{(\log\beta)^\chi}{\beta},
	\]
	uniformly for all $t>0$ and $\beta \ge \e$.
\end{lemma}

\begin{proof} 
	Throughout, define for all $t>0$, $\alpha\in(0\,,1)$, $\beta\ge\e$, and $\chi\ge0$,
	\[
		I(t)=I(t\,,\alpha\,,\beta\,,\chi) = t^\alpha \e^{-\beta t}
		\int_0^t s^{-\alpha}|\log_+(1/s)|^\chi\e^{\beta s}\,\d s.
	\]
	Choose and fix an arbitrary number
	$\beta\ge\e$ throughout.
	If $t\in(0\,,1/\beta)$, then
	\[
		I(t) \le t^\alpha \int_0^t s^{-\alpha} |\log_+(1/s)|^\chi\,\d s
		\lesssim\beta^{-1}(\log\beta)^\chi,
	\]
	where the implied constant depends only on $(\alpha\,,\chi)$. 
	Next we consider the remaining case that $t\ge1/\beta$. In that case, we may change variables 
	to see that
	\[
		I(t) =t \int_0^1 (1-r)^{-\alpha} |\log_+(1/[(1-r)t])|^\chi 
		\e^{-\beta tr}\,\d r.
	\]
	On one hand, because $\log\beta\ge1$ and $\log_+(1/s)\ge1$ for all $s\ge0$,
	\begin{align*}
		& t \int_{1/2}^1 (1-r)^{-\alpha}
			|\log_+(1/[(1-r)t])|^\chi \e^{-\beta tr}\,\d r
			\lesssim t |\log_+(1/t)|^\chi \e^{-\beta t/2} \lesssim \beta^{-1}(\log\beta)^\chi,
	\end{align*}
	where the implied constant depends only on $(\alpha\,,\chi)$.\footnote{If in fact $t\ge1$ then 
	this estimate holds even without the $\log\beta$ since in that case $\log_+(1/t)\lesssim1$.} On the other hand,
	\[
		 t \int_0^{1/2} (1-r)^{-\alpha}
		|\log_+(1/[(1-r)t])|^\chi \e^{-\beta tr}\,\d r
		 \lesssim  t(\log\beta)^\chi  \int_0^\infty
		\e^{-\beta tr}\,\d r =\beta^{-1}(\log\beta)^\chi,
	\]
	where the implied constant depends only on $(\alpha\,,\chi)$.  These estimates
	together prove the lemma.
\end{proof}

Finally, we will need the following two temporal regularity estimates of the Dirichlet heat kernel.
\begin{lemma} \label{lem:temp:log}
	$\sup_{x\in[0,1]}\int_0^1| G_t(x\,,y)-G_s(x\,,y)|\,\d y
	\le \log(t/s)\quad\forall 0<s<t$.
\end{lemma}

\begin{proof}
	Define $\gamma_t(x) = (2\pi t)^{-1/2}\exp\{ -x^2/(2t)\}$
	for all $t>0$ and $x\in\R$,
	and recall that
	$G_t(x\,,y)=\sum_{n=-\infty}^{\infty} \{
	\gamma_t(x-y-2n)-\gamma_t(x+y+2n)\}$, valid for all $t>0$ and $x,y\in[0\,,1]$;
	see for example Dalang and Sanz-Sol\'e 
	\cite[Lemma 1.4.1]{bookDS}. Thus, we may proceed as follows:
	\begin{align*}
		&\int_0^1\left| G_{t}(x\,,y)-G_{s}(x\,,y)\right|\,\d y
			\leq \int_0^1 \sum_{n=-\infty}^{\infty} \vert \gamma_t(x-y-2n)-
			\gamma_s(x-y-2n)\vert \, \d y\\
		&\hskip1.8in+
			\int_0^1 \sum_{n=-\infty}^{\infty} \vert \gamma_t(x+y+2n)-
			\gamma_s(x+y+2n)\vert\, \d y\\
		&\leq \int_0^1 \d y \sum_{n=-\infty}^{\infty} 
			\int_s^t\d r\, \vert\partial_r \gamma_r(x-y-2n)\vert +
			\int_0^1 \d y\,\sum_{n=-\infty}^{\infty}  \int_s^t\d r\,
			\vert\partial_r \gamma_r(x+y+2n))\vert.
	\end{align*}
	It is easy to see that $|\partial_r \gamma_r(z)| \leq r^{-1} \gamma_r(z)$ for all $r>0$ and $z\in\R$.
	Because
	$\sum_{n=-\infty}^\infty\{\gamma_r(x-y-2n)+\gamma_r(x+y+2n)\}$
	defines the heat kernel with Neumann boundary in $[0\,,1]$ -- see for example
	Dalang and Sanz-Sol\'e \cite[Lemma 1.4.3]{bookDS} --
	this concludes the proof.
\end{proof}

\begin{lemma}\label{lem:G-G:L1:t}
	Choose and fix $\delta\in(0\,,1)$ and $\eta\in(\delta\,,1)$. Then,
	\[
		t^\theta\e^{-\beta t}\int_0^t s^{-\theta}
		\e^{\beta s}\,\d s \int_0^1\d y\
		\left| G_{t+\varepsilon-s}(x\,,y) - G_{t-s}(x\,,y) \right| 
		\lesssim \frac{\varepsilon^\delta}{\beta^{1-\eta}},
	\]
	uniformly for all $t\in(0\,,1]$, $\beta>0$, $\varepsilon\in(0\,,1)$, and $x\in[0\,,1]$.
\end{lemma}

\begin{proof}
	Owing to Lemma \ref{lem:temp:log},
	\[
		\int_0^1 \left| G_{t+\varepsilon-s}(x\,,y) - G_{t-s}(x\,,y) \right| \d y
		\lesssim\log\left( 1+ \frac{\varepsilon}{t-s}\right) \le \frac{\varepsilon}{t-s},
	\]
	uniformly for all $x\in[0\,,1]$, $\varepsilon>0$, and $t>s>0$. And if we replace the difference
	in the integral by the corresponding sum, then Lemma \ref{lem:G} yields
	 an upper bound of 2. This shows that, 
	uniformly for all $x\in[0\,,1]$, $\varepsilon>0$, and $t>s>0$,
	\[
		\int_0^1 \left| G_{t+\varepsilon-s}(x\,,y) - G_{t-s}(x\,,y) \right| \d y
		\lesssim 1\wedge \frac{\varepsilon}{t-s},
	\]
	and hence
	\begin{align*}
		& t^\theta\e^{-\beta t}\int_0^t s^{-\theta}
			\e^{\beta s}\,\d s \int_0^1\d y\
			\left| G_{t+\varepsilon-s}(x\,,y) - G_{t-s}(x\,,y) \right| \\
		&\le \varepsilon^\delta t^\theta\e^{-\beta t}\int_0^t s^{-\theta}
			\e^{\beta s}(t-s)^{-\delta}\,\d s
			= \varepsilon^\delta t^{1-\delta}  \int_0^1 r^{-\theta}
			\e^{-\beta t(1-r)}(1-r)^{-\delta}\,\d r\\
		&\lesssim \beta^{-(1-\eta)}
			\varepsilon^\delta t^{\eta-\delta}  \int_0^1 r^{-\theta}
			(1-r)^{-1+\eta-\delta}\,\d r,
	\end{align*}
	valid since $\exp(-y)\lesssim y^{-(1-\eta)}$ uniformly for all $y\ge0$.
	This yields the lemma.
\end{proof}

\section{On convolutions: Bounded case}

In this section we study certain weighted norms of Lebesgue integrals of
random fields. The following identifies some of the necessary notions.

\begin{definition}
	For all real numbers $\alpha,\beta>0$ and $k\ge1$
	define $\cP_{k,\alpha,\beta}$ to be the Banach space
	of all predictable space-time random field $X=\{X(t\,,x)\}_{t>0,x\in[0,1]}$
	such that $\cN_{k,\alpha,\beta,1}(X)<\infty$, where
	the latter norm was defined in \eqref{N}.
\end{definition}
Now let us consider the linear operator $\L$ that maps a space-time function
$f$ to the space-time function $\L_f$ using the following description:\footnote{The
	symbol $\L$ is used to remind that $\L_f$ is a Lebesgue integral.}
\begin{equation}\label{int:L}
	\L_f(t\,,x) = \int_{\sD(t)} G_{t-s}(x\,,y) f(s\,,y)\,\d s\,\d y.
\end{equation}

The following is the main result of this section. 

\begin{proposition}\label{pr:sup:bdd:0}
	For all $\alpha\in(0\,,1)$ and $\bar\alpha>\alpha$
	there exists an otherwise universal number $c=c(\alpha\,,\bar\alpha)>0$
	such that for all $\beta\ge\e$, $k\ge1$, and all $X\in\cP_{k,\alpha,\beta}$,
	\[ 
		\E\left( \|\L_X\|_{\sC_1(\bar\alpha,2\beta)}^k\right) \le
		\left( \frac{c\cN_{k,\alpha,\beta,1}(X)}{\beta^{(1-\delta)/2}}\right)^k.
	\]
\end{proposition}

Proposition \ref{pr:sup:bdd:0} follows immediately from an appeal to Proposition \ref{pr:norm<N}
(see also Remark \ref{rem:norm<N}) using the following lemmas \ref{lem:N(J)}, 
\ref{lem:JX-JX:x}, and \ref{lem:JX-JX:t}. Therefore, we conclude this section, and hence also
the justification for Proposition \ref{pr:sup:bdd:0},
by stating and proving those lemmas next.

\begin{lemma}\label{lem:N(J)}
	Fix real numbers $k\ge 1$, $\alpha\in(0\,,1)$, and $\beta>0$.
	The restriction of the mapping $X\mapsto\L_X$ to the time interval $(0\,,1]$ 
	maps every $\cP_{k,\alpha,\beta}$ to itself
	quasi-isometrically in the sense that
	\[
		\cN_{k,\alpha,\beta,t}(\L_X) \lesssim \frac{\cN_{k,\alpha,\beta,t}(X)}{\beta}
		\quad\forall t\in(0\,,1],
	\]
	and the implied constant depends on $\alpha\in(0\,,1)$ but is otherwise universal.
\end{lemma}

\begin{proof}
	Since $X\mapsto\L_X$ is a linear operator it suffices to consider the case
	that $X$ has continuous sample functions. In that case, predictability 
	can be checked by elementary means, and
	Minkowski's inequality for integrals and Lemma \ref{lem:G} together imply that
	$\|\L_X(r\,,x)\|_k \le 
	\cN_{k,\alpha,\beta,t}(X) 
	\int_0^r s^{-\alpha}\exp(\beta s)\,\d s$ for all $0<r\le t$ and $x \in [0,1]$.
	Lemma \ref{lem:elem} implies the result.
\end{proof}

\begin{lemma}\label{lem:JX-JX:x}
	Choose and fix an arbitrary number $\alpha,\delta\in(0\,,1)$. Then,
	\[
		\sup_{t\in(0,1]}\sup_{x,z\in[0,1]:x\neq z}\left( 
		t^\alpha \e^{-\beta t}\frac{\| \L_X(t\,,x) - \L_X(t\,,z) \|_k}{|x-z|^\delta}
		\right)\lesssim \frac{\cN_{k,\alpha,\beta,1}(X)}{\beta^{1-\delta}},
	\]
	uniformly for all $\beta >0$ and $X\in\cP_{k,\alpha,\beta}$.
\end{lemma}

\begin{proof}
	Choose and fix some $\alpha\in(0\,,1)$.
	For every $\beta,t>0$, $x,z\in[0\,,1]$, and $k\ge1$, 
	\begin{align*}
		&\| \L_X(t\,,x) - \L_X(t\,,z) \|_k \le \int_0^t\d s\int_0^1\d y\
			\left| G_{t-s}(x\,,y) - G_{t-s}(z\,,y) \right| \|X(s\,,y)\|_k\\
		&\le \cN_{k,\alpha,\beta,t}(X)\int_0^t s^{-\alpha} 
			\exp(\beta s)\,\d s \int_0^1\d y\
			\left| G_{t-s}(x\,,y) - G_{t-s}(z\,,y) \right|.
	\end{align*}Appeal to Lemma \ref{lem:G-G:L1:x} to finish.
\end{proof}

\begin{lemma}\label{lem:JX-JX:t}
	Choose and fix arbitrary numbers $\alpha,\delta\in(0\,,1)$
	and $\eta\in(\delta\,,1)$. Then,
	\[
		\sup_{(t,x)\in \sD(1)}
		\left( t^\alpha\e^{-\beta t}
		\frac{\| \L_X(t+\varepsilon\,,x) - \L_X(t\,,x) \|_k}{\varepsilon^\delta} \right) \lesssim
		 \frac{\cN_{k,\alpha,\beta,1}(X)}{\beta^{1-\eta}},
	\]
	uniformly for all $\beta >0$, $k\ge1$,  $\varepsilon\in(0\,,1)$,
	and $X\in\cP_{k,\alpha,\beta}$.
\end{lemma}

\begin{proof}
	We may write
	$\| \L_X(t+\varepsilon\,,x) - \L_X(t\,,x) \|_k \le Q_1 + Q_2,$
	where
	\begin{align*}
		Q_1& = \int_0^t\d s\int_0^1\d y\
			|G_{t+\varepsilon-s}(x\,,y) - G_{t-s}(x\,,y)| \| X(s\,,y)\|_k,\\
		Q_2& = \int_t^{t+\varepsilon}\d s\int_0^1\d y\
			G_{t+\varepsilon-s}(x\,,y) \| X(s\,,y)\|_k.
	\end{align*}
	Because
	\[ 
		Q_1 \le \cN_{k,\alpha,\beta,t}(X)
		\int_0^t s^{-\alpha}\exp(\beta s)\,\d s \int_0^1\d y\
		|G_{t+\varepsilon-s}(x\,,y) - G_{t-s}(x\,,y)|,
	\]
	Lemma \ref{lem:G-G:L1:t} ensures that
	\[
		t^\alpha\e^{-\beta t} Q_1 \lesssim\cN_{k,\alpha,\beta,1}(X)
		\frac{\varepsilon^\delta}{\beta^{1-\eta}},
	\]
	uniformly for all $k\ge1$, $\beta>0$,
	$t\in(0\,,1]$, $\varepsilon\in(0\,,1)$,
	and $x\in[0\,,1]$ and all $X\in\cP_{k,\alpha,\beta}$.	
	And Lemma \ref{lem:G} and Minkowski's inequality together
	ensure that 
	\begin{align*}
		Q_2& \le \cN_{k,\alpha,\beta,1}(X)\int_t^{t+\varepsilon}
			\frac{\exp(\beta s)}{s^\alpha}\,\d s\le 
			\cN_{k,\alpha,\beta,1}(X)\frac{\exp(\beta t)}{t^\alpha}
			\int_t^{t+\varepsilon}
			\exp(-\beta |t-s|)\,\d s\\
		&= \cN_{k,\alpha,\beta,1}(X)\frac{\exp(\beta t)}{\beta t^\alpha}
			\int_0^{\beta\varepsilon}
			\e^{-r}\,\d r\le \cN_{k,\alpha,\beta,1}(X)\frac{\exp(\beta t)}{\beta t^\alpha}
			(\beta\varepsilon\wedge 1),
	\end{align*}
	uniformly for all $k\ge1$, $\beta>0$, $t\in(0\,,1]$, $\varepsilon\in(0\,,1)$,
	and $x\in[0\,,1]$ and all $X\in\cP_{k,\alpha,\beta}$.
	Because $\beta\varepsilon\wedge1\le(\beta\varepsilon)^\delta$, it follows that
	\[
		t^\alpha\e^{-\beta t}Q_2 \le \cN_{k,\alpha,\beta,1}(X)
		\frac{\varepsilon^\delta}{\beta^{1-\delta}},
	\]
	uniformly for all $k\ge1$, $\beta >0$, $t\in(0\,,1]$, $\varepsilon\in(0\,,1)$,
	and $x\in[0\,,1]$ and all $X\in\cP_{k,\alpha,\beta}$. 
	The above estimates for $Q_1$ and $Q_2$ together yield the lemma.
\end{proof}

\section{On stochastic convolutions: Bounded case}

We now study Walsh stochastic integrals of the form,\footnote{The
	symbol $\W$ is used to remind that $\W_X$ is a Walsh stochastic integral.}
\begin{equation}\label{int:W}
	(\W_X)(t\,,x) = \int_{\sD(t)} G_{t-s}(x\,,y) X(s\,,y)\,W(\d s\,\d y),
\end{equation}
where $X$ is a predictable random field that is deterministically bounded; that is,
\begin{equation}\label{bdd}
	\P\left\{ \sup_{(t,x)\in \sD(\infty)}|X(t\,,x)| \le M\right\}=1
	\quad\text{for a deterministic number $M>0$}.
\end{equation}

The following is the main result of this section.

\begin{proposition}\label{pr:sup:bdd}
	For all $\bar\alpha>\alpha>0$
	there exists  $c=c(\alpha\,,\bar\alpha)>0$
	such that for all $\beta,k\ge1$, $M>0$, and all predictable space-time
	random fields $X$ that satisfy \eqref{bdd},
	\[ 
		\E\left( \|\W_X\|_{\sC_\infty(\bar\alpha,2\beta)}^k\right) \le
		\frac{c^k k^{k/2}M^k}{\beta^{\alpha k}}.
	\]
\end{proposition}

Proposition \ref{pr:sup:bdd} follows directly from Proposition \ref{pr:norm<N}
(see also Remark \ref{rem:norm<N}) by applying lemmas \ref{lem:N(I)}, 
\ref{lem:IX-IX:x}, and \ref{lem:IX-IX:t}. Therefore, we end this section by presenting and proving these lemmas below.

\begin{lemma}\label{lem:N(I)}
	Choose and fix an $\alpha>0$. Then, $\W_X\in\cP_{k,\alpha,\beta}$
	for all $k\ge1$, $\beta>0$, and predictable space-time
	random fields $X$ that satisfy \eqref{bdd}. Moreover,
	\[
		\cN_{k,\alpha,\beta,t}(\W_X) 
		\lesssim \frac{M\sqrt{k}}{\beta^{\alpha+\frac14}},
	\]
	uniformly for all $t,\alpha,\beta,M>0$, and $k\in[1\,,\infty)$, and all predictable
	random fields $X$ that satisfy \eqref{bdd}.
\end{lemma}

\begin{proof}
	The measurability properties of $\W_X$ follow from general properties of
	the Walsh stochastic integral; see \cite{Walsh}. It remains to establish
	the stated {\it a priori} bound for the $\cN_{k,\alpha,\beta,T}$-norm of $\W_X$.
	A suitable application of the Burkholder-Davis-Gundy
	inequality (see \cite{minicourse}) yields the following for all real numbers $k\ge1$, $t>0$, and $x\in[0\,,1]$:
	\begin{align*}
		\| (\W_X)(t\,,x) \|_k^2 &\le 4k\int_0^t\d s\int_0^1\d y\
			|G_{t-s}(x\,,y)|^2\|X(s\,,y)\|_k^2\\
		&\le 4kM^2 \int_0^t(t-s)^{-1/2}\,
			\d s \le  8kM^2\sqrt{t};
	\end{align*}
	see Lemma \ref{lem:G}. Take square roots, multiply both sides by $t^\alpha\exp(-\beta t)$
	and optimize over all $(t\,,x)$ to see that
	$\cN_{k,\alpha,\beta,T}(\W_X) \le M\sqrt{8k}
	\sup_{t>0}( t^{\alpha+\frac14}\e^{-\beta t}).$
	This has the desired result.
\end{proof}

\begin{lemma}\label{lem:IX-IX:x}
	Choose and fix a real number $\alpha>0$. Then,
	\[ 
		\sup_{t>0} \sup_{0\le x<z\le1}\left(
		t^\alpha\e^{-\beta t} \frac{\| (\W_X)(t\,,x) - (\W_X)(t\,,z) \|_k }{\sqrt{|x-z|}}
		\right) \lesssim \frac{M\sqrt{k}}{\beta^\alpha},
	\]
	uniformly for all $\beta,M>0$, and $k\ge1$, and all predictable
	random fields $X$ that satisfy \eqref{bdd}.
\end{lemma}

\begin{proof}
	Choose and fix some $\alpha\in(0\,,\frac12)$.
	A suitable application of the Burkholder-Davis-Gundy inequality (see
	\cite{minicourse}) yields the following: Uniformly for all real numbers $t,M>0$, $k\ge1$, and
	$x,z\in[0\,,1]$, and for every predictable space-time random field $X$
	that satisfies \eqref{bdd},
	\begin{align*}
		&\| (\W_X)(t\,,x) - (\W_X)(t\,,z) \|_k^2 \\
			 &\quad\le 4k\int_0^t\d s\int_0^1\d y\
			|G_{t-s}(x\,,y)-G_{t-s}(z\,,y)|^2 \| X(s\,,y)\|_k^2\\
		&\quad  \le4kM^2\int_0^t\d s\int_0^1\d y\
			|G_{t-s}(x\,,y)-G_{t-s}(z\,,y)|^2\lesssim kM^2|x-z|;
	\end{align*}
	see Lemma \ref{lem:G-G:x}. Because $t^\alpha \exp(-\beta t)\lesssim\beta^{-\alpha}$
	uniformly for all $t>0$, this proves the lemma.
\end{proof}

\begin{lemma}\label{lem:IX-IX:t}
	For every fixed $\alpha>0$,
	\[
		\adjustlimits\sup_{t>0}\sup_{\varepsilon\in(0,1)}
		\sup_{x\in[0,1]} \left( t^\alpha\e^{-\beta t}
		\frac{\| (\W_X)(t+\varepsilon\,,x) 
		- (\W_X)(t\,,x) \|_k}{\sqrt\varepsilon}\right)
		\lesssim \frac{M\sqrt{k}}{\beta^\alpha},
	\]
	uniformly for all real numbers $M,\beta>0$ and $k\ge1$, 
	and all predictable space-time random fields $X$ that satisfy \eqref{bdd}.
\end{lemma}

\begin{proof}
	We can write $\| (\W_X)(t+\varepsilon\,,x) - (\W_X)(t\,,x) \|_k\le Q_1+Q_2,$
	where
	\begin{align*}
		Q_1 &= \left\| \int_{(0,t)\times(0,1)} \left[G_{t+\varepsilon-s}(x\,,y) - G_{t-s}(x\,,y)\right]
			X(s\,,y)\, W(\d s\,\d y)\right\|_k,\\
		Q_2 &= \left\| \int_{(t,t+\varepsilon)\times(0,1)} G_{t+\varepsilon-s}(x\,,y)
			X(s\,,y)\, W(\d s\,\d y)\right\|_k.
	\end{align*}
	Identity \eqref{G} and a suitable form of the Burkholder-Davis-Gundy inequality
	together yield the following upper bound for $Q_1^2$:
	\begin{align*}
		&4k\int_0^t\d s\int_0^1\d y\
			\left[G_{t+\varepsilon-s}(x\,,y) - G_{t-s}(x\,,y)\right]^2\|X(s\,,y)\|_k^2\\
		&\le 4kM^2\int_0^t\d s\int_0^1\d y\
			\left[G_{s+\varepsilon}(x\,,y) - G_s(x\,,y)\right]^2\\
		&\le16kM^2\sum_{n=1}^\infty
			\left(1-\e^{-n^2\pi^2 \varepsilon/2}\right)^2
			\int_0^\infty \e^{-n^2\pi^2 s}\d s \\
			&
			\lesssim kM^2\sum_{n=1}^\infty
			\left(\frac{n^2\varepsilon\wedge 1}{n}\right)^2\lesssim kM^2\sqrt\varepsilon,
	\end{align*}
	valid uniformly for all $t,M>0$, $\varepsilon\in(0\,,1)$, $k\ge1$, $x\in[0\,,1]$,
	and predictable $X$ that satisfies \eqref{bdd}. This yields the inequality,
	\begin{equation}\label{Q1}
		\adjustlimits \sup_{t>0}\sup_{\varepsilon\in(0,1)}
		\sup_{x\in[0,1]} t^\alpha\e^{-\beta t}
		 Q_1 \varepsilon^{-1/2}
		\lesssim M\sqrt{k} \sup_{t>0}\left(t^\alpha\e^{-\beta t}\right)
		\propto M\sqrt{k} \beta^{-\alpha},
	\end{equation}
	where the implied constants depend only on $\alpha$.
	
	Similarly, we have the pointwise upper bound for $Q_2^2$:
	\begin{align*}
		& 4k\int_t^{t+\varepsilon}\d s\int_0^1\d y\
			[ G_{t+\varepsilon-s}(x\,,y)]^2\|X(s\,,y)\|_k^2 \\
			&\le4kM^2\int_t^{t+\varepsilon}\d s \int_0^1\d y\
			[ G_{t-s+\varepsilon}(x\,,y)]^2\\
		& \le 4k M^2\int_t^{t+\varepsilon}(t-s+\varepsilon)^{-1/2}\,\d s
			\le 4k M^2\sqrt{\varepsilon}
			\qquad\text{[see Lemma \ref{lem:G}]}.
	\end{align*}
	Thus, \eqref{Q1} continues to hold when $Q_1$ is replaced by $Q_2$.
	Combine the resulting estimates for $Q_1$ and $Q_2$ in order to finish the proof.
\end{proof}

\section{The Lipschitz case with $L^2$ initial data}

Our proof of Theorem \ref{th:main} requires that we first analyze a more general SPDE than
\eqref{SHE:1}. Namely, let us  consider the following generalization of \eqref{SHE:1} in which
the drift and diffusion coefficients are now allowed to be time-dependent functions:
\begin{equation}\label{SHE}
	\partial_t u(t\,,x) = \tfrac12 \partial^2_x u(t\,,x) + b(t\,,u(t\,,x)) + \sigma(t\,,u(t\,,x)) \dot{W}(t\,,x),
\end{equation}
where $(t\,,x)\in\sD(\infty)$, subject to $u(0\,,x) = u_0(x)$, for 
every $x \in [0\,,1]$, and  homogeneous zero-boundary Dirichlet boundary condition.
Throughout this section, we make the following  assumptions on the coefficients:
\begin{assumption}\label{cond-dif2}
	$b,\sigma:(0\,,\infty)\times\R\to\R$ satisfy the following:
	\begin{compactenum}[\indent\indent(1)]
		\item[\textnormal{(1)}] $M_b=\sup_{z\in\R}|b(0\,,z)|<\infty$ and
			there exist $1\le K_b,L_b<\infty$ such that 
			\begin{equation}\label{LKL}
				\lip(b(t))\le K_b+L_b\log_+(1/t)
				\quad\text{uniformly for all $t>0$;}
			\end{equation}
		\item[\textnormal{(2)}] $M_\sigma=\sup_{(t,z)\in(0\,,\infty)\times\R}|\sigma(t\,,z)|$
			and $\lip(\sigma)$ are both finite;
		\item[\textnormal{(3)}] $K_b > L_b\log(8K_b) + M_\sigma^2+ M_b^4.$
	\end{compactenum}
\end{assumption}

\begin{remark}
	The key restriction here is set by \eqref{LKL} and the conditions that
	$M_b\vee M_\sigma\vee\lip(\sigma)<\infty$. 
	The remaining conditions -- namely that $1\le K_b,L_b<\infty$ and (3) holds --
	are assumed without incurring loss in generality
	since we can always choose larger $K_b$ and $L_b$ that still satisfy
	\eqref{LKL}.
\end{remark}

We recall \eqref{D} and that a predictable random field $u$ is a \emph{mild
solution} to \eqref{SHE} when for all $(t\,,x)\in\sD(\infty)$, 
\begin{equation}\label{mild_SHE}
	u(t\,,x) = ( \cG_tu_0)(x) + I_b(t\,,x) + J_\sigma(t\,,x)\quad\text{ a.s.},
\end{equation}
where, $\{\cG_t\}_{t\ge0}$ denotes the heat semigroup, and
\begin{equation}\label{Ib:Js}\begin{split}
	I_b(t\,,x) & = \int_{\sD(t)} G_{t-s}(x\,,y)b(s\,,u(s\,,y))\,\d s\,\d y,\\
	J_\sigma(t\,,x) &  = \int_{\sD(t)} G_{t-s}(x\,,y)\sigma(s\,,u(s\,,y))\,W(\d s\,\d y).
\end{split}\end{equation}
We pause to observe that, in the notation of \eqref{int:L} and \eqref{int:W},
$I_b = \L_{b(\cdot,u)}$ and $J_b=\W_{\sigma(\cdot,u)}$.

This section is naturally divided in four parts. In the first part (\S\ref{sec:lip1})
we establish the existence of a solution to \eqref{SHE} (under the hypotheses
of Assumption \ref{cond-dif2}) and establish an \emph{a priori} energy-type
bound. That is followed by a brief discussion of regularity theory for that solution
(\S\ref{sec:regularity}), then  uniqueness of the solution (\S\ref{sec:unique}),
and finally a stability theorem (\S\ref{sec:stability}) which yields also a comparison 
result. The four subsections follow. 

\subsection{Existence}\label{sec:lip1}
The following is the main result of this subsection.

\begin{theorem} \label{thm_lip}
	If $A>1$ satisfies
	\begin{equation}\label{cond:K}
		L_bK_b^{-1}\log(8AK_b)<1,
	\end{equation} 
	then there exists a predictable random field
	$u=\{u(t\,,x)\}_{(t,x)\in \sD(\infty)}$ that is a mild solution to \eqref{SHE}
	and satisfies the following sub-Gaussian moment bound:
	\[
		\E\left( |u(t\,,x)|^k\right)
		\le (2A/t)^{k/4}\left( \|u_0\|_{L^2[0,1]} +  
		\sqrt{k} \right)^k \e^{4AK_b k t},
	\]
	valid uniformly for all numbers $k\ge1$ and all pairs $(t\,,x)\in\sD(\infty)$.
\end{theorem}

\begin{proof}
	Recall \eqref{D} and let $U_0(t\,,x) = (\cG_t u_0)(x)$ for all $t>0$ and $x\in[0\,,1]$,
	and iteratively define $U_1,U_2,\ldots:\sD(\infty)\to\R$ via
	\begin{equation}\label{picard:n}
		U_{n+1}(t\,,x) = U_0(t\,,x) + I_{b,n}(t\,,x) + J_{\sigma,n}(t\,,x)
		\quad\forall t>0,\ x\in[0\,,1],
	\end{equation}
	where\footnote{In the notation of \eqref{int:L} and \eqref{int:W}, 
	$I_{b,n}=\L_{b(\cdot,U_n)}$ and $J_{\sigma,n} = \W_{\sigma(\cdot,U_n)}$.}
	\begin{equation}\label{Ibn:Jsn}\begin{split}
		I_{b,n}(t\,,x) & = \int_{\sD(t)} G_{t-s}(x\,,y) b(s\,,U_n(s\,,y)) \,\d y \,\d s,\\
		J_{\sigma,n}(t\,,x) & = \int_{\sD(t)} G_{t-s}(x\,,y) \sigma(s\,,U_n(s\,,y)) \, W(\d s\,\d y).
	\end{split}\end{equation}
	[Compare with \eqref{Ib:Js}.]
	In light of \eqref{mild_SHE}, the preceding 
	defines a Picard iteration scheme for the approximation of the solution to our SPDE \eqref{SHE}.
	We study in turn the three quantities on the right-hand side of \eqref{picard:n}.
	
	According to Lemma \ref{lem:G}, 
	$(\int_0^1|G_t(x\,,y)|^2\,\d y)^{1/2}\le t^{-1/4}$ pointwise. Therefore,
	H\"older's inequality yields $|U_0(t\,,x)| \le t^{-1/4} \|u_0\|_{L^2[0,1]}$	
	for all $t>0$ and $x\in[0\,,1]$. In light of \eqref{N}, it follows that
	\begin{equation}\label{N(U0)}
		\mathcal{N}_{k,\frac14,\beta,T}(U_0) \le\|u_0\|_{L^2[0,1]}
		\quad\forall \beta>0,\ k\ge1.
	\end{equation}
	This is the desired estimate for the first term in \eqref{picard:n}.
	
	We now consider the second term on the right-hand side of
	\eqref{picard:n}. Thanks to Assumption \ref{cond-dif2},
	$|b(t\,,z)|\le M_b + |z|\{ K_b + L_b\log_+(1/t)\}$ for all 
	$t>0$ and $z\in\R$. Therefore,
	for all $k\ge 1$, $n\in\Z_+$, $t\in(0\,,T]$, and $x\in[0\,,1]$,
	\begin{align*}
		&\|I_{b,n}(t\,,x)\|_k \le\int_0^t\d s\int_0^1\d y\
			G_{t-s}(x\,,y)\|b(s\,, U_n(s\,,y))\|_k\\
		&\le M_b\int_0^t\d s\int_0^1\d y\
			G_{t-s}(x\,,y) + K_b\int_0^t\d s\int_0^1\d y\
			G_{t-s}(x\,,y)\| U_n(s\,,y)\|_k \\
		&\qquad+  L_b\int_0^t\d s\int_0^1\d y\
			G_{t-s}(x\,,y)\log_+(1/s)\| U_n(s\,,y)\|_k\\
		&\le M_b t+ K_b\cN_{k,\frac14,\beta,T}(U_n)\int_0^t s^{-\frac14} \e^{-\beta s}\,\d s\\
		&\qquad 
			+ L_b\cN_{k,\frac14,\beta,T}(U_n)\int_0^t
			s^{-\frac14}\log_+(1/s)\e^{-\beta s}\,\d s.
	\end{align*}
	Since $\log_+(a)\ge1$ for all $a\ge0$, and 
	$\sup_{t>0}[t^{1+\frac14} \exp(-\beta t)]\propto \beta^{-(1+\frac14)}<\beta^{-\frac14}$,
	uniformly for all $\beta\ge\e$, it follows from
	Lemma \ref{lem:rv} that, 
	uniformly for every $k\ge1$, $\beta\ge\e$, and $n\in\Z_+$,
	\begin{equation}\label{N(Ibn)}
		\cN_{k,\frac14,\beta,T}(I_{b,n}) \lesssim
		M_b\beta^{-\frac14}  + \left[ K_b\beta^{-1} 
		+ L_b\beta^{-1}\log\beta\right]
		\cN_{k,\frac14,\beta,T}(U_n),
	\end{equation}
	where the implied constant is universal. The preceding is the desired estimate for the second
	term on the right-hand side of \eqref{picard:n}.
	
	The third quantity in \eqref{picard:n} is estimated already by Lemma \ref{lem:N(I)}
	as follows: For every $k\ge 1$, $\beta\ge\e$, and $n\in\Z_+$,
	\begin{equation}\label{N(Jsn)}
		\cN_{k,\frac14,\beta,T}(J_{\sigma,n}) \lesssim M_\sigma \sqrt{k/\beta},
	\end{equation}
	and note that the implied constant is universal.
	
	Combine \eqref{N(U0)}, \eqref{N(Ibn)}, and \eqref{N(Jsn)}
	in order to conclude that there exists a constant $C\ge1$ such that 
	\begin{align}\nonumber
		\cN_{k,\frac14,\beta,T}(U_{n+1})
			&\le \|u_0\|_{L^2[0,1]} + C M_b\beta^{-\frac14}  \\
			&+ C\left[ K_b\beta^{-1} 
			+ L_b\beta^{-1}\log\beta\right]
			\cN_{k,\frac14,\beta,T}(U_n)
		 + CM_\sigma \sqrt{k/\beta},
			\label{recur:U}
	\end{align}
	uniformly for all $k\ge 1$, $\beta\ge\e$, and $n\in\Z_+$. 
	
	The inequality \eqref{recur:U} holds for all $\beta\ge\e$; we need to now choose 
	$\beta$ (approximately) optimally. That choice critically depends on the relative
	sizes of the various loose parameters $M_b$, $K_b$, $L_b$, and $k$. For
	our purposes, $K_b$ should be viewed as the largest of those loose parameters;
	see Assumption \ref{cond-dif2} and \eqref{cond:K}.
	Therefore, we now choose and fix $\beta$ as follows in order to minimize the effect of
	$K_b$ on the size of $U_{n+1}$:
	\begin{equation}\label{beta}
		\beta=4CK_b.
	\end{equation} 
	Assumption \ref{cond-dif2} ensures, among other things, that $K_b\ge1$ whence
	$\beta\ge4>\e$, so the preceding estimates are applicable for this choice of $\beta$.
	For this particular choice of $\beta$, the recursion \eqref{recur:U} simplifies to the following: 
	\begin{align*}
		\cN_{k,\frac14,\beta,T}(U_{n+1})
			&\le \|u_0\|_{L^2[0,1]} + A M_bK_b^{-\frac14}  + 
			\tfrac14\left[L_bK_b^{-1}\log(AK_b) + 1\right]
			\cN_{k,\frac14,\beta,T}(U_n)\\
		&\hskip.5in + A M_\sigma \sqrt{k/K_b},
	\end{align*}
	for a constant $A\ge C$. The  number $A$
	was announced in the statement of Theorem \ref{thm_lip}, and thanks to 
	the condition \eqref{cond:K}, the preceding implies that, 
	with the same parameter dependencies as before, 
	\begin{equation}\label{N(U)<N(U)}
		\cN_{k,\frac14,\beta,T}(U_{n+1})
		\le B +\tfrac12 \cN_{k,\frac14,\beta,T}(U_n),
	\end{equation}
	where $B=\|u_0\|_{L^2[0,1]} + A M_bK_b^{-1/4}  + 
	A M_\sigma \sqrt{k/K_b}.$ The preceding display is an iterative inequality
	indexed by $n$, and can be solved, thanks to \eqref{N(U0)}, in order to yield 
	$\limsup_{n\to\infty}\cN_{k,\frac14,\beta,T}(U_n) \le 2B$. Thanks to
	\eqref{N}, the latter is another way to state the following:
	Uniformly for all $t\in(0\,,T]$, $x\in[0\,,1]$, and $k\ge1$,
	\begin{equation}\label{moment}\begin{split}
		&\limsup_{n\to\infty}\E\left( |U_n(t\,,x)|^k\right)
			\le (2B)^k t^{-k/4}  \exp(\beta kt)\\
		&\le 2^k t^{-k/4}\left( \|u_0\|_{L^2[0,1]} + A M_b K_b^{-1/4}  + 
			A M_\sigma \sqrt{k/K_b} \right)^k \e^{4AK_b kt}\\
		&\le (2A)^k t^{-k/4}\left( \|u_0\|_{L^2[0,1]} + 1   + 
			\sqrt{k} \right)^k \e^{4AK_b kt};
	\end{split}\end{equation}
	valid thanks to Part (3) of Assumption \ref{cond-dif2}, 
	\eqref{beta}, and the fact that $A\ge C\ge1$. 
	Because the constants do not depend on $T$, \eqref{moment}
	is in fact true for all $t>0$.
	Next we prove that $u(t\,,x)=\lim_{n\to\infty}U_n(t\,,x)$ exists in $L^k(\Omega)$
	for every $(t\,,x)\in\sD(T)$ and $k\ge1$. That, and a standard argument,
	which we skip, together imply that $u$ solves \eqref{SHE} up to time $t=T$. Moreover,
	\eqref{moment} implies the \emph{a priori} $L^k$-estimate
	of the theorem, thanks to Fatou's lemma. Thus, it remains to 
	prove that $\{U_n(t\,,x)\}_{n=1}^\infty$ is Cauchy in $L^k(\Omega)$ 
	for every $(t\,,x)\in\sD(\infty)$ and $k\ge1$. In order to do that, we replicate
	and adapt the calculation that led to \eqref{N(U)<N(U)} in order to find that,
	as long as $\beta$ is still defined by \eqref{beta},
	condition \eqref{cond:K} yields
	\begin{equation}\label{N(U-U)}
		\cN_{k,\frac14,\beta,T}(U_{n+1}-U_n)\le\tfrac12\cN_{k,\frac14,\beta,T}(U_n-U_{n-1})
		\qquad\forall n\in\N,\ T>0.
	\end{equation}
	We omit the details as they are very close to the details of
	the proof of \eqref{N(U)<N(U)}. Instead we note that, as a result, 
	$\sum_{n=1}^\infty\cN_{k,\frac14,\beta,T}(U_n-U_{n-1})<\infty$ for every $k\ge1$.
	Thank to \eqref{N}, this  implies the pointwise existence of $u(t\,,x)=\lim_{n\to\infty} U_n(t\,,x)$,
	where the limit holds in $L^k(\Omega)$ for every $k\ge1$.
	
	Embedded within this argument lies also the fact that, for the same choice of $\beta$ as in
	\eqref{beta}, and for all $k\ge1$ and $T>0$,
	\[
		\lim_{n\to\infty}\cN_{k,\frac14,\beta,T}(I_{b,n}-I_b)=
		\lim_{n\to\infty}\cN_{k,\frac14,\beta,T}(J_{\sigma,n}-J_\sigma)=0.
	\]
	See also \eqref{Ib:Js} and \eqref{Ibn:Jsn}.
	Therefore, \eqref{picard:n}, \eqref{mild_SHE}, and \eqref{mild_SHE:1} together imply  the 
	conclusion of the proof.
\end{proof}

\subsection{Regularity}\label{sec:regularity}
In this section we continue the discussion of \S\ref{sec:lip1} from which the notation
of the present section is derived as well. As the title of the section might suggest,
we now study the solution $u$ to \eqref{SHE}, which is an extension of the
SPDE \eqref{SHE:1} to the setting in which $b$ and $\sigma$ can be time-dependent.
Throughout this section, Assumption \ref{cond-dif2} is assumed as well, just as it was
in \S\ref{sec:lip1}.

Let us begin with the following well-known result, which is an immediate
consequence of \eqref{G} and Bessel's inequality.

\begin{lemma}\label{lem:IC}
	For every $t>0$, $\lim_{t\to0+}\cG_t u_0= u_0$  in $L^2[0\,,1]$.
\end{lemma}

Recall the norms $\|\,\cdots\|_{\sC_1(\alpha\,,\beta)}$ from \eqref{sec:Banach}. 
The main result of this subsection is
the following \emph{a priori} bound. 

\begin{proposition}\label{pr:lip2}
	In the context of Theorem \ref{thm_lip}, $u$ is continuous on $(0\,,\infty)\times[0\,,1]$,
	and for every fixed $\alpha>\frac14$,
	\[
		\left\{\E\left( \| u \|_{\sC_1(\alpha,8AK_b)}^k\right)\right\}^{1/k}\lesssim
		\|u_0\|_{L^2[0,1]} + \sqrt k,
	\]
	where the implied constant does not depend on $(u_0,k)
	\in L^2([0\,,1])\times[1\,,\infty)$, nor does it depend on $(b\,,\sigma)$
	that satisfy Assumption \ref{cond-dif2} and \eqref{cond:K}.
\end{proposition}

\begin{proof}
	It follows readily from \eqref{N} and Theorem \ref{thm_lip} that
	\begin{equation}\label{N(u)}
		\cN_{k,\frac14,4AK_b,1}(u)
		\le (2A)^{1/4}\left( \|u_0\|_{L^2[0,1]} + \sqrt{k} \right),
	\end{equation}
	for all $\beta\ge 4AK_b$, $k\ge1$.
	Now, consider \eqref{N(Ibn)}, and let $n\to\infty$ in both inequalities in
	order to obtain the following:
	For every $k\ge1$, and $\beta\ge\e$, 
	\begin{equation}\label{N(I_b)}\begin{split}
		\cN_{k,\frac14,\beta,1}(I_b) &\lesssim M_b\beta^{-1/4}  + \left[ K_b\beta^{-1} 
			+ L_b\beta^{-1}\log\beta\right] \cN_{k,\frac14,\beta,1}(u),\\
		&\lesssim M_b\beta^{-1/4}  + \left[ K_b\beta^{-1} 
			+ L_b\beta^{-1}\log\beta\right] \left( \|u_0\|_{L^2[0,1]} + \sqrt{k} \right),
	\end{split}\end{equation}
	thanks to \eqref{N(u)} and the fact that $\beta\mapsto\cN_{k,\frac14,\beta,1}(u)$
	is non increasing; see \eqref{N}. We also pause to emphasize that 
	the implied constants in \eqref{N(I_b)} do not depend $(u_0,k\,,\beta)
	\in L^2([0\,,1])\times[1\,,\infty)\times[4AK_b\,,\infty)$, nor on $(b\,,\sigma)$
	that satisfy Assumption \ref{cond-dif2} and \eqref{cond:K}.
	
	Next, let us apply \eqref{N(u)} together with Lemma \ref{lem:JX-JX:x} to see that,
	because $4AK_b\ge 4$,
	\begin{align}\nonumber
		\sup_{t\in(0,1]}\sup_{x,z\in[0,1]:x\neq z}\left( 
			t^{1/4} \e^{-\beta t}\frac{\| I_b(t\,,x) - I_b(t\,,z) \|_k}{|x-z|^{1/4}}
			\right) & \lesssim \frac{\cN_{k,\frac14,\beta,1}(I_b)}{\beta^{3/4}}\\
		&\le\cN_{k,\frac14,\beta,1}(I_b),
			\label{Ib-Ib:x}
	\end{align}
	where the implied constants do not depend $(u_0,k\,,\beta)
	\in L^2([0\,,1])\times[1\,,\infty)\times[4AK_b\,,\infty)$, nor on $(b\,,\sigma)$
	that satisfy Assumption \ref{cond-dif2} and \eqref{cond:K}.
	Likewise, Lemma \ref{lem:JX-JX:t} yields the following:
	\begin{equation}\label{Ib-Ib:t}
		\sup_{\varepsilon\in(0,1)}\sup_{(t,x)\in \sD(1)}
		\left( t^\alpha\e^{-\beta t}
		\frac{\| I_b(t+\varepsilon\,,x) - I_b(t\,,x) \|_k}{\varepsilon^{1/4}} \right) 
		\lesssim \cN_{k,\frac14,\beta,1}(I_b),
	\end{equation}
	where the implied constants do not depend $(u_0,k\,,\beta)
	\in L^2([0\,,1])\times[1\,,\infty)\times[4AK_b\,,\infty)$, nor on $(b\,,\sigma)$
	that satisfy Assumption \ref{cond-dif2} and \eqref{cond:K}.

	Thanks to \eqref{Ib-Ib:x} and \eqref{Ib-Ib:t}, we may
	apply Proposition \ref{pr:norm<N}; see also Remark \ref{rem:norm<N}.
	Because of that proposition and the particular form of the constant $L$ in 
	\eqref{L}, applied with $\alpha=\frac14$, $\bar\alpha=\kappa$,
	$\beta=4AK_b$, and $\bar\beta=8AK_b$,  and thanks to one or two 
	back-to-back appeals to Assumption \ref{cond-dif2}
	and \eqref{cond:K},  it follows from \eqref{N(u)} that for every $\kappa>1/4$,
	\begin{equation}\label{norm:Ib}\begin{split}
		&\left\{\E\left( \| I_b \|_{\sC_1(\kappa,8AK_b)}^k\right)\right\}^{1/k}\lesssim
			\cN_{k,\frac14,8AK_b,1}(I_b)\\\
		&\lesssim \frac{M_b}{K_b^{1/4}}  + \left[ 1
			+ \frac{L_b}{K_b}\log(8AK_b)\right] \left( \|u_0\|_{L^2[0,1]} + \sqrt{k} \right)
			\lesssim \|u_0\|_{L^2[0,1]} + \sqrt{k},
	\end{split}\end{equation}
	where the implied constants do not depend $(u_0,k)
	\in L^2([0\,,1])\times[1\,,\infty)$, nor on $(b\,,\sigma)$
	that satisfy Assumption \ref{cond-dif2} and \eqref{cond:K}.
	
	Next, consider \eqref{N(Jsn)},
	and let $n\to\infty$ in both inequalities in
	order to obtain the following:
	\begin{equation}\label{N(IJ)}
		\cN_{k,\frac14,\beta,1}(J_\sigma)
		\lesssim M_\sigma \sqrt{k/\beta},
	\end{equation}
	where again the implied constant does not depend $(u_0,k\,,\beta)
	\in L^2([0\,,1])\times[1\,,\infty)\times[4AK_b\,,\infty)$, nor on $(b\,,\sigma)$
	that satisfy Assumption \ref{cond-dif2} and \eqref{cond:K}.
	We may apply \eqref{N(IJ)} together with Lemmas \ref{lem:IX-IX:x} 
	and  \ref{lem:IX-IX:t} in order to see that,
	\begin{align}
		&\sup_{t\in(0,1]}\sup_{x,z\in[0,1]:x\neq z}\left( 
			t^{1/4} \e^{-\beta t}\frac{\| J_\sigma(t\,,x) - J_\sigma(t\,,z) \|_k}{\sqrt{|x-z|}}
			\right)\lesssim M_\sigma\frac{\sqrt{k}}{\beta^{1/4}},
			\label{Js-Js:x}\\
			\intertext{and}
		&\sup_{\varepsilon\in(0,1)}\sup_{(t,x)\in \sD(1)}
			\left( t^{1/4}\e^{-\beta t}
			\frac{\| J_\sigma(t+\varepsilon\,,x) - J_\sigma(t\,,x) \|_k}{\sqrt\varepsilon} \right) \lesssim
			M_\sigma\frac{\sqrt{k}}{\beta^{1/4}},
			\label{Js-Js:t}
	\end{align}
	where the implied constants do not depend $(u_0,k)
	\in L^2([0\,,1])\times[1\,,\infty)$, nor on $(b\,,\sigma)$
	that satisfy Assumption \ref{cond-dif2} and \eqref{cond:K}.
	
	Proposition \ref{pr:norm<N}, used in conjunction with \eqref{N(IJ)}, \eqref{Js-Js:x},
	and \eqref{Js-Js:t}, yields the following: For every fixed $\kappa>\frac14$,
	\begin{equation}\label{norm:Js}
		\left\{\E\left( \| J_\sigma \|_{\sC_1(\kappa,8AK_b)}^k\right)\right\}^{1/k}\lesssim
		\sqrt k,
	\end{equation}
	where the implied constants do not depend $(u_0,k)
	\in L^2([0\,,1])\times[1\,,\infty)$, nor on $(b\,,\sigma)$
	that satisfy Assumption \ref{cond-dif2} and \eqref{cond:K}.
	
	Finally, Lemma \ref{lem:G} and the Cauchy-Schwarz inequality together imply
	that $|(\cG_t u_0)(x)|\le \|u_0\|_{L^2[0,1]}t^{-1/4}$
	for every $t>0$ and $x\in[0\,,1]$,
	and hence
	\begin{equation}\label{norm:Gu}
		\adjustlimits\inf_{\alpha\ge\frac14}\sup_{\beta>0}\|\cG u_0\|_{\sC_1(\alpha,\beta)}
		=\sup_{\beta>0}\|\cG u_0\|_{\sC_1(\frac14,\beta)}\le \|u_0\|_{L^2[0,1]}.
	\end{equation}
	The proposition follows from applying the triangle inequality to \eqref{mild_SHE:1}
	using \eqref{norm:Ib}, \eqref{norm:Js}, and
	\eqref{norm:Gu}.
\end{proof}

\subsection{Uniqueness}\label{sec:unique}
For every space-time random field $X=\{X(t\,,x)\}_{t>0,x\in[0,1]}$ we 
may define
\begin{equation}\label{M}
	\cM_t (X) = \sup_{s\in(0,t]}\sup_{x\in[0,1]}
	\left[ s^{1/4}\|X(s\,,x)\|_2\right]
	\qquad\forall t>0.
\end{equation}
Thanks to \eqref{N},
\[
	\e^{-\beta t}\cM_t(X) \le \cN_{2,\frac14,\beta,t}(X) \le \cM_t(X)
	\qquad\forall t,\beta>0.
\]
Since the \emph{a priori} moment estimate of Theorem \ref{thm_lip}
can be recast as follows
\[
	\sup_{t>0}\cN_{2,\frac14,4AK_b,t}(u) \le 2A\left(\|u_0\|_{L^2[0,1]}+\sqrt 2\right),
\]
we see that $\cM_t(u)<\infty$ for all $t>0$. The following is a uniqueness type of converse.

\begin{lemma}\label{lem:unique}
	In the context of Theorem \ref{thm_lip} -- see also Assumption \ref{cond-dif2}
	and \eqref{cond:K} -- suppose that $\tilde{u}$ is a random-field
	solution to \eqref{SHE} subject to initial data $u_0\in L^2[0\,,1]$, and such that
	$\cM_t(\tilde{u})<\infty$ for all $t>0$. Then, $\tilde{u}$ is a modification of $u$.
\end{lemma}

\begin{proof}[Sketch of proof.]
	Since $\cN_{2,\frac14,\beta,t}(\tilde{u})\le \cM_t(\tilde{u})<\infty$
	for all $t>0$, we can repeat the same argument that led to \eqref{N(U-U)} in order to see that,
	for the same $\beta$ as in \eqref{beta},
	\[
		\cN_{2,\frac14,\beta,t}(u-\tilde{u})\le
		\tfrac12\cN_{2,\frac14,\beta,t}(u-\tilde{u})
		\qquad\forall t>0,
	\]
	whence $\cN_{2,\frac14,\beta,t}(u-\tilde{u})=0$ for every $t>0$. 
	This proves the lemma.
\end{proof}

\subsection{Stability}\label{sec:stability}
If $u_0,\tilde{u}_0\in L^2[0\,,1]$ then Theorem \ref{thm_lip}
assures us of the existence of a solution $u$ to \eqref{SHE} with initial data $u_0$,
and also of a solution $\tilde{u}$ to \eqref{SHE} with initial data $\tilde{u}_0$.
\emph{Stability} is the assertion that $u(t\,,x)\approx \tilde{u}(t\,,x)$
for every $(t\,,x)$
when $u_0\approx \tilde{u}_0$. The following
is a stability type statement that is modest and can be improved upon, but
is good enough for our purposes.
Recall the notation in \eqref{M}, and also that Assumption \ref{cond-dif2}
and \eqref{cond:K} are assumed to hold in this section. 

\begin{lemma}\label{lem:stability}
	For every $T>0$ there exists a number $c=c(T,b\,,\sigma)>0$,
	whose value does not depend on  $(u_0\,,\tilde{u}_0)$,
	such that $\cM_T( u - \tilde{u}) \le c\|u_0-\tilde{u}_0\|_{L^2[0,1]}.$
\end{lemma}

\begin{proof}
	We may write $u$ via \eqref{mild_SHE} and \eqref{Ib:Js}. 
	Similarly, we can represent $\tilde{u}$ as follows:
	\[
		\tilde{u}(t\,,x) = ( \cG_t\tilde{u}_0)(x) + \tilde{I}_b(t\,,x) 
		+ \tilde{J}_\sigma(t\,,x)\quad\text{ a.s.},
	\]
	where, $\{\cG_t\}_{t\ge0}$ denotes the heat semigroup, and
	\begin{align*}
		\tilde{I}_b(t\,,x) & = \int_{\sD(t)} G_{t-s}(x\,,y)
			b(s\,,\tilde{u}(s\,,y))\,\d s\,\d y,\\
		\tilde{J}_\sigma(t\,,x) &  = \int_{\sD(t)} G_{t-s}(x\,,y)
			\sigma(s\,,\tilde{u}(s\,,y))\,W(\d s\,\d y).
	\end{align*}
	In this way, we have the natural decomposition,
	\[
		\cN_{k,\frac14,\beta,T} (u - \tilde{u})  \le Q_1 + Q_2 + Q_3,
	\]
	for every $k\ge1$ and $\beta,T>0$, where
	\begin{gather*}
		Q_1 = \sup_{t\in(0,T]}\sup_{x\in[0,1]}\left[
			t^{1/4}\e^{-\beta t}| ( \cG_t u_0)(x) - 
			( \cG_t\tilde{u}_0)(x)|\right],\\
		Q_2 = \cN_{2,\frac14,\beta,T} (I_b - \tilde{I}_b),\qquad
		 Q_3 = \cN_{2,\frac14,\beta,T} (J_\sigma - \tilde{J}_\sigma).
	\end{gather*}
	We estimate $Q_1,Q_2,Q_3$ in turn.
	
	Thanks to the Cauchy-Schwarz inequality and \eqref{G_tf},
	\begin{equation}\label{Q1:s}\begin{split}
		Q_1 &\le \sup_{t\in(0,T]}\sup_{x\in[0,1]}\left[
			t^{1/4}\e^{-\beta t}\|G_t\|_{L^2[0,1]}\right]
			\|u_0-\tilde{u}_0\|_{L^2[0,1]}\\
		&\le\|u_0-\tilde{u}_0\|_{L^2[0,1]},
	\end{split}\end{equation}
	since $\|G_t\|_{L^2[0,1]}\le t^{-1/4}$ [Lemma \ref{lem:G}] and $\exp(-\beta t)\le1$.
	This is the desired estimate for $Q_1$.
	
	Now we work toward bounding $Q_2$. For all $t,\beta>0$ and $x\in[0\,,1]$,
	\begin{align*}
		&\| I_b(t\,,x) - \tilde{I}_b(t\,,x) \|_2 \le \int_{\sD(t)}
			G_{t-s}(x\,,y)\lip(b(s))\| u(s\,,y)-\tilde{u}(s\,,y)\|_2\,\d s\,\d y\\
		&\lesssim \int_{\sD(t)} G_{t-s}(x\,,y) \log_+(1/s)\| u(s\,,y)-\tilde{u}(s\,,y)\|_2\,\d s\,\d y
			\quad\text{[Assumption \ref{cond-dif2}]}\\
		&\le\cN_{2,\frac14,\beta,t}(u-\tilde{u}) \int_0^t s^{-1/4} \e^{\beta s} \log_+(1/s) \,\d s,
	\end{align*}
	where the implied constant does not depend on $(t\,,x)$. Therefore,
	we multiply both sides by $t^{1/4}\exp(-\beta t)$ and optimize over $(t\,,x\,,\beta)$ in order
	to deduce from Lemma \ref{lem:elem} [with $\delta=1/2$] that, uniformly for all $\beta,T>0$,
	\begin{equation}\label{Q2:s}
		Q_2 \lesssim \beta^{-1/2} \cN_{2,\frac14,\beta,T}(u-\tilde{u}).
	\end{equation}
	
	Finally, we can apply the Walsh isometry for stochastic integrals 
	in order to see that, for every $t,\beta>0$ and $x\in[0\,,1]$,
	\begin{align*}
		&\| J_\sigma(t\,,x) - \tilde{J}_\sigma(t\,,x) \|_2^2 \le 
			[\lip(\sigma)]^2 \int_{\sD(t)}
			[G_{t-s}(x\,,y)  ]^2 \| u(s\,,y)-\tilde{u}(s\,,y)\|_2^2\,\d s\,\d y\\
		&\lesssim [\cN_{2,\frac14,\beta,t}(u-\tilde{u})]^2
			\int_0^t s^{-1/2}\e^{2\beta s}\,\d s\int_0^1\d y\ |G_{t-s}(x\,,y)|^2\\
		&\le [\cN_{2,\frac14,\beta,t}(u-\tilde{u})]^2
			\int_0^t  s^{-1/2} (t-s)^{-1/2} \e^{2\beta s} \,\d s
			\quad\text{[Lemma \ref{lem:G}]}\\
		&= \e^{2\beta t} [\cN_{2,\frac14,\beta,t}(u-\tilde{u})]^2
			\int_0^1 r^{-1/2} (1-r)^{-1/2} \e^{-2\beta t(1-r)}\,\d r,
	\end{align*}
	where the implied constant is independent of $(t\,,x\,,\beta)$. Since
	$\exp(-x)\lesssim x^{-1/4}$ uniformly for all $x>0$, it follows from
	the above that
	\[
		\| J_\sigma(t\,,x) - \tilde{J}_\sigma(t\,,x) \|_2
		\lesssim (\beta t)^{-1/8}\e^{\beta t} \cN_{2,\frac14,\beta,t}(u-\tilde{u}),
	\]
	Multiply both sides by $t^{1/4}$ and optimize over $(t\,,x)$ in order to see that
	\begin{equation}\label{Q3:s}
		Q_3 \lesssim (T/\beta)^{1/4}\cN_{2,\frac14,\beta,T}(u-\tilde{u}),
	\end{equation}
	uniformly for all $\beta>0$. 
 
 	Finally, combine \eqref{Q1:s}, \eqref{Q2:s}, and \eqref{Q3:s} in order to see that
	\[
		\cN_{2,\frac14,\beta,T}(u-\tilde{u}) \le \|u_0-\tilde{u}_0\|_{L^2[0,1]}
		+ c\beta^{-1/4}\
		\cN_{2,\frac14,\beta,T}(u-\tilde{u}),
	\]
	where $c=c(T,b\,,\sigma) $, independently of the value of $\beta>0$.
	This yields 
	\[
		\cN_{2,\frac14,16c,T}(u-\tilde{u}) \le 2\|u_0-\tilde{u}_0\|_{L^2[0,1]},
	\]
	which in turn implies that $\cM_T(u-\tilde{u})\le 2\exp(16cT)\|u_0-\tilde{u}_0\|_{L^2[0,1]}$,
	as desired.
\end{proof}

\begin{corollary}\label{cor:comparison}
	If $u_0,\tilde{u}_0\in L^2[0\,,1]$ satisfy
	$u_0 \le \tilde{u}_0$, then $\P\{u\le\tilde{u}\}=1$.
\end{corollary}

\begin{proof}
	For every $\varepsilon>0$ and $x\in[0\,,1]$ let 
	\[
		u_{0,\varepsilon}(x) = (\cG_\varepsilon u_0)(x)
		\quad\text{and}\quad
		\tilde{u}_{0,\varepsilon}(x) = (\cG_\varepsilon \tilde{u}_0)(x).
	\]
	Note that $u_{0,\varepsilon}$ and $\tilde{u}_{0,\varepsilon}$ are bounded [Lemma \ref{lem:G}]
	and measurable functions. Let $u_\varepsilon$ and $\tilde{u}_\varepsilon$ 
	respectively denote the solutions to \eqref{SHE} 
	with respective initial data $u_{0,\varepsilon}$ and $\tilde{u}_{0,\varepsilon}$. By the comparison
	theorem of Gei\ss\ and Manthey \cite{GM}, 
	\begin{equation}\label{eq:GM}
		\P\{u_\varepsilon\le \tilde{u}_\varepsilon\}=1.
	\end{equation}
	Thanks to Lemmas \ref{lem:IC} and \ref{lem:stability}, for every $t>0$ fixed,
	$\|u_\varepsilon(t) - \tilde{u}_\varepsilon(t)\|_{L^2[0,1]}\to 0$ in probability as $\varepsilon\to0+$.
	By Fubini's theorem, there exists a Lebesgue-null set $\mathfrak{N}\subset[0\,,1]$ --
	independently of the trajectories of the processes $u_\varepsilon$ and $\tilde{u}_\varepsilon$ --
	such that,
	for every $x\in[0\,,1]\setminus\mathfrak{N}$,
	$u_\varepsilon(t\,,x)\to u(t\,,x)$ in probability as $\varepsilon\to0+$.
	Therefore, \eqref{eq:GM} implies that
	\begin{equation}\label{eq:GM1}
		\P\{u(t\,,x)\le \tilde{u}(t\,,x)\}=1
		\quad\text{for all $t>0$ and $x\in[0\,,1]\setminus\mathfrak{N}$.}
	\end{equation}
	Since $[0\,,1]\setminus\mathfrak{N}$ is necessarily dense in $[0\,,1]$,
	the corollary follows from \eqref{eq:GM1} and the continuity of $u$ and $\tilde{u}$;
	see Theorem \ref{thm_lip}.
\end{proof}

\section{Proof of Theorem \ref{th:main}}

With the technical results of the previous sections under way,
the remainder of the proof is divided into a few steps, the first two of which are
patterned after the stopping-time arguments of
Dalang, Khoshnevisan, and Zhang \cite{DKZ}.

Throughout, we choose and fix a number
$\alpha\in(\frac14\,,1]$. Then, we
define for all functions $b:\R\to\R$,
$N\ge\e$, and $\alpha\in(\frac14\,,1)$, and 
$(t\,,z)\in (0\,,1]\times\R$, a space-time function $b_N:(0\,,\infty)\times\R\to\R$
via
\begin{equation}\label{b_N}
	b_N(t\,,z)=\begin{cases}
		b(N/t^\alpha) \quad &\text{if } \quad z>N/t^\alpha \\
		b(z) \quad &\text{if } \quad \vert z \vert \leq N/t^\alpha\\
		b(-N/t^\alpha) \quad &\text{if } \quad z<-N/t^\alpha.
	\end{cases}
\end{equation}
In principle we should really write $b_{N,\alpha}$ instead of $b_N$.
Because $\alpha$ is fixed, this notational omission should not cause
confusion.

{\bf Step 1.}
Choose and fix some $N\ge\e$.
In the first step of the proof, we shall 
consider the following special case: There exist  $\theta_1, \theta_2\ge0$
such that
\begin{equation}\label{btt}
	b(z)=\theta_1 + \theta_2 |z|\log_+|z|\qquad\forall z\in\R.
\end{equation}
For every $N\ge\e$ and $t>0$,
\begin{align*}
	\partial_z b(z) &= \theta_2 \log_+z + \frac{z}{z+\e}
		\le \theta_2\log_+(N/t^\alpha)+1
		\qquad\forall z\in(0\,,N/t^\alpha)\\
	&\le\theta_2\log_+(N) +1 + \theta_2\log_+(1/t^\alpha),
\end{align*}
where we recall $\alpha\in(\frac14\,,1)$ is held fixed.
Additionally, $\theta_2\log_+(N)+1 \le 2\theta_2\log N+1\le 2(\theta_2+1)\log N$
(with room to spare)
since $N\ge\e$, and $\log_+(1/t^\alpha)\le 2\log_+(1/t)$ since $\alpha<1$.\footnote{%
Indeed, $\log_+(1/t^\alpha)\le\log_+(1/t)$ when $t\le1$ and 
$\log_+(1/t^\alpha)\le\log(1+\e)\le 2\le2\log_+(1/t)$ when $t>1$.}
These observations, and the symmetry of $b$, together show that every $b_N$ 
satisfies Assumption \ref{cond-dif2} with  
\[ 
	M_{b_N} = |b(0)|,\quad
	K_{b_N} = 2(\theta_2+1)\log N,
	\quad\text{and}\quad
	L_{b_N} = 2.
\]
Because $A>1$ is a universal constant -- see Theorem \ref{thm_lip} --
there exists $N_0=N_0(\theta_2) \ge \e$
such that \eqref{cond:K} is equivalent to $N> N_0$. Therefore, we may apply
Theorem \ref{thm_lip} in order to see that, for every $N>N_0$, 
the SPDE \eqref{SHE} -- with $b$ replaced  by  $b_N$ and subject to 
initial data $u_0$ -- has a continuous, random-field solution $u_N$.
Choose and fix some $\alpha\in(\frac14\,,1)$. 
Proposition \ref{pr:lip2} ensures that $u_N$ satisfies
\begin{equation}\label{tight}
	\left\{ \E\left( \| u_N\|_{\sC_1(\alpha,16A(\theta_2+1)\log N)}^k \right)\right\}^{1/k}\lesssim
	(1+\|u_0\|_{L^2[0,1]}) \sqrt k,
\end{equation}
where the implied constant does not depend on $(N\,,k)\in(N_0\,,\infty)\times[1\,,\infty)$.
It is easy to see that \eqref{tight} is equivalent to the assertion that  
the random variable $\| u_N\|_{\sC_1(\alpha,16A(\theta_2+1)\log N)}$ is sub-Gaussian and hence
\begin{equation}\label{tight_1}
	\exists\lambda=\lambda(\|u_0\|_{L^[0,1]})>0:\
	\sup_{N>N_0}
	\E\exp\left( \lambda \| u_N\|_{\sC_1(\alpha,16A(\theta_2+1)\log N)}^2 \right)\le1.
\end{equation}

Next, we define the stopping times
\[
	T_N = \inf\left\{ t\in(0\,,1):\ \|u_N(t)\|_{C[0,1]} > N/t^\alpha \right\}
	\qquad\forall N>N_0.
\]
We may consistently define the random field
\begin{equation}\label{uuN}
	u(t\,,x) = u_N(t\,,x)\quad\forall (t\,,x)\in ( 0\,, T_N ) \times [0\,,1],
\end{equation}
regardless of the value of $N>N_0$, since $b_N(u_N(t))=b(u_N(t))$  for all $t<T_N$.
Proposition \ref{pr:lip2} implies that $u$ is continuous on $(0\,,T_N]\times[0\,,1]$.
Also, 
\[
	T_N = \inf\left\{ t\in(0\,,1):\ \|u(t)\|_{C[0,1]}> N/t^\alpha \right\}
	\qquad\forall N>N_0,
\]
where $\inf\varnothing=\infty$.
Define 
\begin{equation}\label{t0}
	t_0 = \frac{1}{32A(\theta_2+1)},
\end{equation}
and observe that $t_0 < 1/32 < 1$.
Because
\begin{equation}\label{NNNN}
	 \| u\|_{\sC_t(\alpha,16A(\theta_2+1)\log N)} \ge  N^{-1/2}
	 \sup_{s\in(0,t]}\left( s^\alpha \|u(s)\|_{C[0,1]}\right)
	 \qquad\forall t\in(0\,,t_0],
\end{equation}
we can now see that
\[
	\forall t\in(0\,,1),\ N>N_0:\quad
	T_N(\omega) \le t \Rightarrow 
	\| u\|_{\sC_t(\alpha,16A(\theta_2+1)\log N)}(\omega)\ge \sqrt N.
\]
Therefore, \eqref{tight_1} and Chebyshev's inequality together imply that
\begin{equation}\label{TNt}
	\P\{T_N  \le t_0\} \le \exp(-\sqrt N)\quad\forall N>N_0,
\end{equation}
whence $\lim_{N\to\infty}\P\{T_N>t_0\}=1$.
In this way we have proved that $u$ solves \eqref{SHE}, with initial profile $u_0$ and  up to time $t_0$.
This is another way to say that the existence assertion of the theorem is true when $b$
has the form \eqref{btt}, with $t_0$ given by \eqref{t0}.\\

{\bf Step 2.}
We now study the existence of a solution in 
the general case that $(b\,,\sigma)$ satisfies Assumption \ref{cond-dif2}
and $b$ satisfies \eqref{cond:K}. It might help to recall also that the constant $A$ in
\eqref{cond:K} is universal; see Theorem \ref{thm_lip}. Because $|b(z)|=\mathcal{O}(|z|\log|z|)$
as $|z|\to\infty$, we can find $\theta_1,\theta_2\ge0$ such that
\begin{equation}\label{-B<b<B}
	-B \le b\le B\text{ pointwise on $\R$, where }
	B(z)=\theta_1+\theta_2|z|\log_+|z|\quad\forall z\in\R.
\end{equation}
Recall $b_N$ from \eqref{b_N} and define $B_N$ in exactly the same way but replace ``$b$''
by ``$B$'' everywhere in \eqref{b_N}. Let $U_N^\pm$ denote the (continuous) solution to \eqref{SHE}
where $b$ is replaced by $\pm B_N$, and $u_N$ the  (continuous) solution to \eqref{SHE} where $b$
is replaced by $b_N$. The existence and uniqueness of the random fields $u_N$, $U_N^+$,
and $U_N^-$ are ensured by Theorem \ref{thm_lip}. Because of \eqref{-B<b<B} 
and Corollary \ref{cor:comparison}, the following holds almost surely:
\[
	U_N^-(t\,,x) \le u_N(t\,,x) \le U_N^+(t\,,x)\quad\forall t>0,\ x\in[0\,,1].
\]
Because
\[
	\|u_N\|_{\sC_1(\alpha,16A(\theta_2+1))\log N}\le
	\|U_N^-\|_{\sC_1(\alpha,16A(\theta_2+1))\log N}+
	\|U_N^+\|_{\sC_1(\alpha,16A(\theta_2+1))\log N},
\]
and the last two $\sC_1$-norms have moments that are bounded uniformly in $N$ (see Step 1),
the present choice of $u_N$ also satisfies \eqref{tight}. Now we may reapply the argument
of Step 1 to finish the short-time existence of $u$, which can be defined exactly as in \eqref{uuN}.
This completes the short-time existence of a solution $u$ in Theorem \ref{th:main}. It remains
to verify that this solution satisfies properties (1)--(3) of the theorem, and that it has the
said uniqueness property.\\

\textbf{Step 3.} Let $t_0$ be any fixed and nonrandom time by which $u$ solves 
\eqref{SHE}. In Step 1, we gave a formula for $t_0$ in the case that $b$ is of the
form \eqref{btt}. The argument of Step 2 also yields a formula (the time
$t_0$ for $u$ can be selected as the minimum of the two $t_0$s for 
$U^\pm=\lim_{N\to\infty}U_N^\pm$). Because $u$ is a mild solution to \eqref{SHE},
we can reuse the same moment methods that yielded continuity in Step 1 (that is
when $b$ has the form \eqref{btt}) in order
to derive the continuity of $u$ in the present general case from the Kolmogorov continuity
theorem; see \eqref{Ib-Ib:x}, \eqref{Ib-Ib:t}, \eqref{Js-Js:x}, \eqref{Js-Js:t} (the deterministic
term in \eqref{mild_SHE} is always continuous for $(t\,,x)\in(0\,,\infty)\times[0\,,1]$
thanks to the dominated convergence theorem). Assertion (1) of Theorem \ref{th:main} follows.\\

\textbf{Step 4.}
Next we prove Assertion (2) of theorem. Thanks to Step 2, it suffices to prove
(2) in the case that $b$ has the form \eqref{btt}, which we assume 
to be the case for the remainder of Step 4.
Choose and fix some $\alpha\in(\frac14\,,1)$.
Thanks to Fatou's lemma, \eqref{tight_1}, and the estimation of $u$ by 
$U^\pm =\lim_{N\to\infty} U_N^\pm$ in Step 2, there exists
$\lambda=\lambda(u_0)\in(0\,,1)$ such that
\[
	\E\exp\left( \lambda \|u\|_{\sC_1(\alpha,16A(\theta_2+1)\log N)}^2;\
	T_N>t_0\right)\le 1.
\]
Therefore, \eqref{t0}, \eqref{TNt}, and the preceding together yield the following: 
\[
	\P\left\{ \|u\|_{\sC_{t_0} (\alpha,\frac12\log N)} \ge q \right\}
	\le \e^{-\lambda q^2} + \e^{-\sqrt N}\qquad\forall q>0, N>N_0.
\]
Because of \eqref{NNNN}, this yields
\[
	\P\left\{ \sup_{s\in(0,t_0]}(s^\alpha\|u(s)\|_{C[0,1]}) \ge q\sqrt N \right\}
	\le \e^{-\lambda q^2} + \e^{-\sqrt N}\qquad\forall q>0, N>N_0.
\]
We now choose $q=N^{1/4}$ in order to see that $\sup_{s\in(0,t_0]}(s^\alpha\|u(s)\|_{C[0,1]})^{2/3}$
has exponential moments and in particular 
$\limsup_{t\to 0+} t^\alpha \|u(t)\|_{C[0,1]}<\infty$ a.s.
This shows that 
$t^\beta \|u(t)\|_{C[0,1]}\to 0$ as $t\downarrow0$
whenever $\beta\in(\alpha\,,1)$. The latter fact completes
the proof of Step 4 since, by choosing $\alpha$ sufficiently close to $\frac14$,
we can ensure that $\beta$ can be any number in $(\frac14\,,1)$.\\

\textbf{Step 5.} We now prove assertion (3) of Theorem \ref{th:main}; that is, 
\[
	\int_0^1|u(t\,,x)-u_0(x)|^2\,\d x\xrightarrow[(t\to0+)]{\P\,}0.
\]
Thanks to the stopping-time argument of Step 2, it suffices to prove the result
when $b$ has the form \eqref{btt} and, for the same $N_0$ that appeared in Step 1,
\[
	\exists N>N_0:\quad \lim_{t\to0+} \E\int_0^1|u_N(t\,,x)-u_0(x)|^2\,\d x=0.
\]
Choose and fix a non random $N>N_0$.
We can write $u_N(t)=\cG_t u_0 + I_{b,N} + J_{\sigma,N}$, in parallel with
\eqref{mild_SHE}. We have already observed in Lemma \ref{lem:IC} that
$\cG_t u_0\to u_0$ in $L^2[0\,,1]$ as $t\to0+$. It remains to prove
that $\|I_{b,N}(t)\|_{L^2([0,1]\times\Omega)}\to0$ and
$\|J_{\sigma,N}(t)\|_{L^2([0,1]\times\Omega)}\to0$, both 
as $t\to0+$. We start with a proof of
the latter. By the Walsh isometry for stochastic integrals, and thanks to the boundedness
assumption on $\sigma$ and Lemma \ref{G},
\[
	\|J_{\sigma,N}(t)\|_{L^2([0,1]\times\Omega)}^2=\E\left( \|J_{\sigma,N}(t) \|_{L^2[0,1]}^2\right)
	\le M_\sigma^2\int_0^t\d s\int_0^1\d y\ |G_{t-s}(x\,,y)|^2
	\lesssim\sqrt t,
\]
for all $t>0$. Clearly, the preceding tends to zero as $t\to0+$. Furthermore,
we may apply the triangle inequality to find that, uniformly for all $t\in(0\,,1/\e)$ [say],
\begin{align*}
	&\|I_{b,N}(t) \|_{L^2([0,1]\times\Omega)}\le\int_0^t\d s\int_0^1\d y\
		\left(\int_0^1\d x\ |G_{t-s}(x\,,y)|^2\right)^{1/2} \|b_N(u_N(s\,,y))\|_2\\
	&\le2K_b\int_0^t \log_+(1/s) \,\d s\int_0^1\d y\
		(t-s)^{-1/4} \|u_N(s\,,y)\|_2\\
	&\lesssim \int_0^t s^{-1/4}
		(t-s)^{-1/4}\log_+(1/s) \,\d s \lesssim t^{1/2}\log(1/t),
\end{align*}
thanks to the \emph{a priori} estimate in Theorem \ref{thm_lip}.
This establishes part (3) of Theorem \ref{th:main}. It remains to verify the
uniqueness portion of the theorem. \\

\textbf{Step 6.}
We conclude the proof by establishing the promised uniqueness
statement of the theorem.
Let $v$ denote another continuous random-field solution on the time interval $(0\,,t_0]$
such that $\sup_{s\in(0,t_0]}(s^\alpha\|v(s)\|_{C[0,1]})<\infty$ a.s.\ for some $\alpha>\frac14$.
Define
\[
	S_N=\inf\{s\in(0\,,t_0]:\, \|v(s)\|_{C[0,1]}>N/s^\alpha\}
	\qquad\forall N\ge1,
\]
where $\inf\varnothing=\infty$, and note that
\begin{equation}\label{SNt}
	\P\{S_N <\infty\} = \P\left\{
	\sup_{s\in(0,t_0]}(s^\alpha\|v(s)\|_{C[0,1]})\ge N\right\} \to0\quad\text{as $N\to\infty$}.
\end{equation}
The uniqueness of $u_N$ (Lemma \ref{lem:unique}) ensures that
$u(s)= u_N(s)=v(s)$ for all $s\in(0\,,t_0\wedge S_N\wedge T_N)$. Thanks to 
\eqref{TNt} and \eqref{SNt}, $u(s)=v(s)$ for all $s\in(0\,,t_0)$ a.s.
This completes the proof of Theorem \ref{th:main}.
\qed





\bibliography{L201}

\def\cprime{$'$}
\begin{bibdiv}
\begin{biblist}

\bib{Bass}{book}{
      author={Bass, Richard~F.},
       title={Diffusions and {E}lliptic {O}perators},
      series={Probability and its Applications (New York)},
   publisher={Springer-Verlag, New York},
        date={1998},
        ISBN={0-387-98315-5},
      review={\MR{1483890}},
}

\bib{Cerrai}{article}{
      author={Cerrai, Sandra},
       title={Stochastic reaction-diffusion systems with multiplicative noise
  and non-{L}ipschitz reaction term},
        date={2003},
        ISSN={0178-8051,1432-2064},
     journal={Probab. Theory Related Fields},
      volume={125},
      number={2},
       pages={271\ndash 304},
         url={https://doi.org/10.1007/s00440-002-0230-6},
      review={\MR{1961346}},
}

\bib{CDHS}{article}{
      author={Chen, Le},
      author={Foondun, Mohammud},
      author={Huang, Jingyu},
      author={Salins, Michael},
       title={Global solution for superlinear stochastic heat equation on
  {$\Bbb R^d$} under {O}sgood-type conditions},
        date={2025},
        ISSN={0951-7715,1361-6544},
     journal={Nonlinearity},
      volume={38},
      number={5},
       pages={Paper No. 055026, 25},
         url={https://doi.org/10.1088/1361-6544/add3b2},
      review={\MR{4904500}},
}

\bib{ChenHuang}{article}{
      author={Chen, Le},
      author={Huang, Jingyu},
       title={Superlinear stochastic heat equation on {$\Bbb{R}^d$}},
        date={2023},
        ISSN={0002-9939,1088-6826},
     journal={Proc. Amer. Math. Soc.},
      volume={151},
      number={9},
       pages={4063\ndash 4078},
         url={https://doi.org/10.1090/proc/16436},
      review={\MR{4607649}},
}

\bib{dPZ}{book}{
      author={Da~Prato, Giuseppe},
      author={Zabczyk, Jerzy},
       title={Stochastic {E}quations in {I}nfinite {D}imensions},
     edition={Second Edition},
      series={Encyclopedia of Mathematics and its Applications},
   publisher={Cambridge University Press, Cambridge},
        date={2014},
      volume={152},
        ISBN={978-1-107-05584-1},
         url={https://doi.org/10.1017/CBO9781107295513},
      review={\MR{3236753}},
}

\bib{minicourse}{book}{
      author={Dalang, Robert},
      author={Khoshnevisan, Davar},
      author={Mueller, Carl},
      author={Nualart, David},
      author={Xiao, Yimin},
      editor={Rassoul-Agha, Firas},
       title={A {M}inicourse on {S}tochastic {P}artial {D}ifferential
  {E}quations},
      series={Lecture Notes in Mathematics},
   publisher={Springer-Verlag, Berlin},
        date={2009},
      volume={1962},
        ISBN={978-3-540-85993-2},
        note={Held at the University of Utah, Salt Lake City, UT, May 8--19,
  2006},
      review={\MR{1500166}},
}

\bib{Dalang}{article}{
      author={Dalang, Robert~C.},
       title={Extending the martingale measure stochastic integral with
  applications to spatially homogeneous s.p.d.e.'s},
        date={1999},
        ISSN={1083-6489},
     journal={Electron. J. Probab.},
      volume={4},
       pages={no. 6, 29},
         url={https://doi.org/10.1214/EJP.v4-43},
      review={\MR{1684157}},
}

\bib{DKN}{article}{
      author={Dalang, Robert~C.},
      author={Khoshnevisan, Davar},
      author={Nualart, Eulalia},
       title={Hitting probabilities for systems of non-linear stochastic heat
  equations with additive noise},
        date={2007},
        ISSN={1980-0436},
     journal={ALEA Lat. Am. J. Probab. Math. Stat.},
      volume={3},
       pages={231\ndash 271},
      review={\MR{2365643}},
}

\bib{DKZ}{article}{
      author={Dalang, Robert~C.},
      author={Khoshnevisan, Davar},
      author={Zhang, Tusheng},
       title={Global solutions to stochastic reaction-diffusion equations with
  super-linear drift and multiplicative noise},
        date={2019},
        ISSN={0091-1798,2168-894X},
     journal={Ann. Probab.},
      volume={47},
      number={1},
       pages={519\ndash 559},
         url={https://doi.org/10.1214/18-AOP1270},
      review={\MR{3909975}},
}

\bib{bookDS}{book}{
      author={Dalang, Robert~C.},
      author={Sanz-Sol\'e, Marta},
       title={{S}tochastic {P}artial {D}ifferential {E}quations, {S}pace-{T}ime
  {W}hite {N}oise and {R}andom {F}ields},
   publisher={Springer Monographs in Mathematics, Springer, New York},
        date={To appear in 2026},
}

\bib{BonderGroisman}{article}{
      author={Fern\'andez~Bonder, Julian},
      author={Groisman, Pablo},
       title={Time-space white noise eliminates global solutions in
  reaction-diffusion equations},
        date={2009},
        ISSN={0167-2789,1872-8022},
     journal={Phys. D},
      volume={238},
      number={2},
       pages={209\ndash 215},
         url={https://doi.org/10.1016/j.physd.2008.09.005},
      review={\MR{2516340}},
}

\bib{FK}{article}{
      author={Foondun, Mohammud},
      author={Khoshnevisan, Davar},
       title={Intermittence and nonlinear parabolic stochastic partial
  differential equations},
        date={2009},
        ISSN={1083-6489},
     journal={Electron. J. Probab.},
      volume={14},
       pages={no. 21, 548\ndash 568},
         url={https://doi.org/10.1214/EJP.v14-614},
      review={\MR{2480553}},
}

\bib{FKN}{article}{
      author={Foondun, Mohammud},
      author={Khoshnevisan, Davar},
      author={Nualart, Eulalia},
       title={Instantaneous everywhere-blowup of parabolic {SPDE}s},
        date={2024},
        ISSN={0178-8051,1432-2064},
     journal={Probab. Theory Related Fields},
      volume={190},
      number={1-2},
       pages={601\ndash 624},
         url={https://doi.org/10.1007/s00440-024-01263-7},
      review={\MR{4797376}},
}

\bib{FKN2}{article}{
      author={Foondun, Mohammud},
      author={Khoshnevisan, Davar},
      author={Nualart, Eulalia},
       title={On the well-posedness of {SPDE}s with locally {L}ipschitz
  coefficients},
        date={2025},
     journal={arXiv preprint},
      volume={2411.09381},
         url={https://arxiv.org/abs/2411.09381},
        note={preprint},
}

\bib{FN2021}{article}{
      author={Foondun, Mohammud},
      author={Nualart, Eulalia},
       title={The {O}sgood condition for stochastic partial differential
  equations},
        date={2021},
        ISSN={1350-7265,1573-9759},
     journal={Bernoulli},
      volume={27},
      number={1},
       pages={295\ndash 311},
         url={https://doi.org/10.3150/20-BEJ1240},
      review={\MR{4177371}},
}

\bib{FN2022}{article}{
      author={Foondun, Mohammud},
      author={Nualart, Eulalia},
       title={Non-existence results for stochastic wave equations in one
  dimension},
        date={2022},
        ISSN={0022-0396,1090-2732},
     journal={J. Differential Equations},
      volume={318},
       pages={557\ndash 578},
         url={https://doi.org/10.1016/j.jde.2022.02.038},
      review={\MR{4387900}},
}

\bib{FP}{article}{
      author={Foondun, Mohammud},
      author={Parshad, Rana~D.},
       title={On non-existence of global solutions to a class of stochastic
  heat equations},
        date={2015},
        ISSN={0002-9939,1088-6826},
     journal={Proc. Amer. Math. Soc.},
      volume={143},
      number={9},
       pages={4085\ndash 4094},
         url={https://doi.org/10.1090/proc/12036},
      review={\MR{3359596}},
}

\bib{G}{inproceedings}{
      author={Garsia, Adriano~M.},
       title={Continuity properties of {G}aussian processes with
  multidimensional time parameter},
        date={1972},
   booktitle={Proceedings of the {S}ixth {B}erkeley {S}ymposium on
  {M}athematical {S}tatistics and {P}robability ({U}niv. {C}alifornia,
  {B}erkeley, {C}alif., 1970/1971), {V}ol. {II}: {P}robability {T}heory},
   publisher={Univ. California Press, Berkeley, CA},
       pages={369\ndash 374},
      review={\MR{410880}},
}

\bib{GM}{article}{
      author={Gei\ss, Christel},
      author={Manthey, Ralf},
       title={Comparison theorems for stochastic differential equations in
  finite and infinite dimensions},
        date={1994},
        ISSN={0304-4149,1879-209X},
     journal={Stochastic Process. Appl.},
      volume={53},
      number={1},
       pages={23\ndash 35},
         url={https://doi.org/10.1016/0304-4149(94)90055-8},
      review={\MR{1290705}},
}

\bib{HanKim}{article}{
      author={Han, Beom-Seok},
      author={Kim, Kyeong-Hun},
       title={Boundary behavior and interior {H}\"older regularity of the
  solution to nonlinear stochastic partial differential equation driven by
  space-time white noise},
        date={2020},
        ISSN={0022-0396,1090-2732},
     journal={J. Differential Equations},
      volume={269},
      number={11},
       pages={9904\ndash 9935},
         url={https://doi.org/10.1016/j.jde.2020.07.002},
      review={\MR{4122650}},
}

\bib{HanYi}{article}{
      author={Han, Beom-Seok},
      author={Yi, Jaeyun},
       title={{$L_ p$}-regularity theory for the stochastic reaction-diffusion
  equation with super-linear multiplicative noise and strong dissipativity},
        date={2024},
        ISSN={0022-0396,1090-2732},
     journal={J. Differential Equations},
      volume={379},
       pages={569\ndash 599},
         url={https://doi.org/10.1016/j.jde.2023.10.042},
      review={\MR{4660204}},
}

\bib{Krylov}{article}{
      author={Krylov, Nicolai},
       title={A brief overview of the {$L_p$}-theory of {SPDE}s},
        date={2008},
        ISSN={0321-3900},
     journal={Theory Stoch. Process.},
      volume={14},
      number={2},
       pages={71\ndash 78},
      review={\MR{2479735}},
}

\bib{LSZ}{article}{
      author={Li, Yue},
      author={Shang, Shijie},
      author={Zhai, Jianliang},
       title={Large deviation principle for stochastic reaction-diffusion
  equations with superlinear drift on {$\Bbb R$} driven by space-time white
  noise},
        date={2024},
        ISSN={0894-9840,1572-9230},
     journal={J. Theoret. Probab.},
      volume={37},
      number={4},
       pages={3496\ndash 3539},
         url={https://doi.org/10.1007/s10959-024-01345-1},
      review={\MR{4805051}},
}

\bib{MiaoYuan}{article}{
      author={Miao, Changxing},
      author={Yuan, Baoquan},
       title={Solutions to some nonlinear parabolic equations in pseudomeasure
  spaces},
        date={2007},
        ISSN={0025-584X,1522-2616},
     journal={Math. Nachr.},
      volume={280},
      number={1-2},
       pages={171\ndash 186},
         url={https://doi.org/10.1002/mana.200410472},
      review={\MR{2290390}},
}

\bib{Mueller1991}{article}{
      author={Mueller, Carl},
       title={Long time existence for the heat equation with a noise term},
        date={1991},
        ISSN={0178-8051,1432-2064},
     journal={Probab. Theory Related Fields},
      volume={90},
      number={4},
       pages={505\ndash 517},
         url={https://doi.org/10.1007/BF01192141},
      review={\MR{1135557}},
}

\bib{Mueller98}{article}{
      author={Mueller, Carl},
       title={Long-time existence for signed solutions of the heat equation
  with a noise term},
        date={1998},
        ISSN={0178-8051,1432-2064},
     journal={Probab. Theory Related Fields},
      volume={110},
      number={1},
       pages={51\ndash 68},
         url={https://doi.org/10.1007/s004400050144},
      review={\MR{1602036}},
}

\bib{Mueller2000}{article}{
      author={Mueller, Carl},
       title={The critical parameter for the heat equation with a noise term to
  blow up in finite time},
        date={2000},
        ISSN={0091-1798,2168-894X},
     journal={Ann. Probab.},
      volume={28},
      number={4},
       pages={1735\ndash 1746},
         url={https://doi.org/10.1214/aop/1019160505},
      review={\MR{1813841}},
}

\bib{MuellerSowers}{article}{
      author={Mueller, Carl},
      author={Sowers, Richard},
       title={Blowup for the heat equation with a noise term},
        date={1993},
        ISSN={0178-8051,1432-2064},
     journal={Probab. Theory Related Fields},
      volume={97},
      number={3},
       pages={287\ndash 320},
         url={https://doi.org/10.1007/BF01195068},
      review={\MR{1245247}},
}

\bib{PSZ}{article}{
      author={Pan, Tianyi},
      author={Shang, Shijie},
      author={Zhang, Tusheng},
       title={Large deviations of stochastic heat equations with logarithmic
  nonlinearity},
        date={2025},
        ISSN={0926-2601,1572-929X},
     journal={Potential Anal.},
      volume={62},
      number={2},
       pages={439\ndash 463},
         url={https://doi.org/10.1007/s11118-024-10142-8},
      review={\MR{4866619}},
}

\bib{QuittnerSouplet}{book}{
      author={Quittner, Pavol},
      author={Souplet, Philippe},
       title={Superlinear {P}arabolic {P}roblems},
      series={Birkh\"auser Advanced Texts: Basler Lehrb\"ucher. [Birkh\"auser
  Advanced Texts: Basel Textbooks]},
   publisher={Birkh\"auser Verlag, Basel},
        date={2007},
        ISBN={978-3-7643-8441-8},
        note={Blow-up, global existence and steady states},
      review={\MR{2346798}},
}

\bib{Salins2021}{article}{
      author={Salins, Michael},
       title={Global solutions for the stochastic reaction-diffusion equation
  with super-linear multiplicative noise and strong dissipativity},
        date={2022},
        ISSN={1083-6489},
     journal={Electron. J. Probab.},
      volume={27},
       pages={Paper No. 12, 17},
         url={https://doi.org/10.1214/22-ejp740},
      review={\MR{4372099}},
}

\bib{Salins2022}{article}{
      author={Salins, Michael},
       title={Global solutions to the stochastic reaction-diffusion equation
  with superlinear accretive reaction term and superlinear multiplicative noise
  term on a bounded spatial domain},
        date={2022},
        ISSN={0002-9947,1088-6850},
     journal={Trans. Amer. Math. Soc.},
      volume={375},
      number={11},
       pages={8083\ndash 8099},
         url={https://doi.org/10.1090/tran/8763},
      review={\MR{4491446}},
}

\bib{Salins2025}{article}{
      author={Salins, Michael},
       title={Solutions to the stochastic heat equation with polynomially
  growing multiplicative noise do not explode in the critical regime},
        date={2025},
        ISSN={0091-1798,2168-894X},
     journal={Ann. Probab.},
      volume={53},
      number={1},
       pages={223\ndash 238},
         url={https://doi.org/10.1214/24-aop1704},
      review={\MR{4852006}},
}

\bib{SZ}{article}{
      author={Shang, Shijie},
      author={Zhang, Tusheng},
       title={Stochastic heat equations with logarithmic nonlinearity},
        date={2022},
        ISSN={0022-0396,1090-2732},
     journal={J. Differential Equations},
      volume={313},
       pages={85\ndash 121},
         url={https://doi.org/10.1016/j.jde.2021.12.033},
      review={\MR{4362370}},
}

\bib{Walsh}{incollection}{
      author={Walsh, John~B.},
       title={{A}n {I}ntroduction to {S}tochastic {P}artial {D}ifferential
  {E}quations},
        date={1986},
   booktitle={\'{E}cole d'\'et\'e{} de {P}robabilit\'es de {S}aint-{F}lour,
  {XIV}---1984},
      series={Lecture Notes in Math.},
      volume={1180},
   publisher={Springer, Berlin},
       pages={265\ndash 439},
         url={https://doi.org/10.1007/BFb0074920},
      review={\MR{876085}},
}

\end{biblist}
\end{bibdiv}

\end{document}